\documentclass{article}
\RequirePackage{etex}

\usepackage[margin=1in]{geometry}
\usepackage[english]{babel}
\usepackage[utf8]{inputenc}
\usepackage{subcaption}
\usepackage{amsmath}
\usepackage{csquotes}
\usepackage{amssymb}
\usepackage{amsfonts}
\usepackage{amsthm}
\usepackage{mathrsfs}
\usepackage[all]{xy}
\usepackage[pdftex]{graphicx}
\usepackage{color}
\usepackage{url}
\usepackage{indent first}
\usepackage[labelfont=bf,labelsep=period,justification=raggedright]{caption}
\usepackage[english]{babel}
\usepackage[utf8]{inputenc}
\usepackage{hyperref}
\usepackage[colorinlistoftodos]{todonotes}
\usepackage{tkz-fct}
\usepackage{tikz}
\usetikzlibrary{calc}
\usepackage{multicol}
\PassOptionsToPackage{dvipsnames,svgnames}{xcolor}
\usepackage{textcomp}

\DeclareMathOperator{\Hom}{Hom}

\DeclareMathOperator{\Spec}{Spec}
\DeclareMathOperator{\Proj}{Proj}

\renewcommand{\AA}{\mathbb{A}}
\newcommand{\CC}{\mathbb{C}}

\newcommand{\GG}{\mathbb{G}}

\newcommand{\NN}{\mathbb{N}}

\newcommand{\PP}{\mathbb{P}}

\newcommand{\RR}{\mathbb{R}}
\newcommand{\VV}{\mathbb{V}}
\newcommand{\ZZ}{\mathbb{Z}}

\newcommand{\mcB}{\mathcal{B}}
\newcommand{\mcE}{\mathcal{E}}
\newcommand{\mcF}{\mathcal{F}}
\newcommand{\mcG}{\mathcal{G}}

\newcommand{\mcL}{\mathcal{L}}
\newcommand{\mcM}{\mathcal{M}}
\newcommand{\mcN}{\mathcal{N}}
\newcommand{\mcO}{\mathcal{O}}

\newcommand{\mcU}{\mathcal{U}}

\newcommand{\mfm}{\mathfrak{m}}

\def\multiset#1#2{\ensuremath{\left(\kern-.3em\left(\genfrac{}{}{0pt}{}{#1}{#2}\right)\kern-.3em\right)}}

\theoremstyle{plain}
\newtheorem{thm}{Theorem}
\newtheorem{cor}[thm]{Corollary}

\newtheorem{prop}[thm]{Proposition}
\newtheorem{lemma}[thm]{Lemma}

\theoremstyle{definition}
\newtheorem{defn}[thm]{Definition}

\newtheorem*{notn}{Notation}

\theoremstyle{remark}
\newtheorem{rem}[thm]{Remark}
\newtheorem*{ex}{Example}

\numberwithin{equation}{section}
\numberwithin{thm}{section}

\usepackage{arxiv}

\usepackage[utf8]{inputenc} 
\usepackage[T1]{fontenc}    
\usepackage{hyperref}       
\usepackage{url}            
\usepackage{booktabs}       
\usepackage{amsfonts}       
\usepackage{nicefrac}       
\usepackage{microtype}      
\usepackage{lipsum}
\usepackage{titlesec}

\titleformat{\subsection}[runin]
{\normalfont\normalsize\bfseries}{\thesubsection}{1em}{}
\titleformat{\subsubsection}[runin]
{\normalfont\normalsize\bfseries}{\thesubsubsection}{1em}{}

\title{On the deformation to the normal cone in Arakelov geometry }

\usepackage[style=alphabetic]{biblatex}
\author{Dorian Ni}

\begin{document}

\maketitle

\begin{abstract}
We present an Arakelov theoretic version of the deformation to the normal cone. In particular, the geometric data is enriched with a deformation of a Hermitian line bundle. We introduce numerical invariants called arithmetic Hilbert invariants and prove the conservation of these invariants along the deformation. \\
In a following article, this conservation of number theorem will allow a demonstration of the arithmetic Hilbert-Samuel theorem.
\end{abstract}


\section{Introduction}

\subsection{Conservation of number.} One of the leading principle in algebraic geometry in the attempts to define the notion of intersection multiplicity is the conservation of number principle. It is one of the two pillars of Schubert calculus for computing intersection numbers in \cite{schubert} and states that \emph{the number of solutions of an enumerative problem should be invariant by deformation}. In particular, it culminates in \cite{fulton} by the use of the deformation to the normal cone to define the intersection product. This deformation is a flat deformation over $\PP^1$ from a closed immersion $i : Y \to X$ to the zero section immersion into the normal cone $s_Y : Y \to N_YX$. It unified the algebraic point of view on intersection multiplicities via Hilbert polynomials as defined by Samuel in \cite{samuel} and the continuity point of view given by the conservation of number principle.\\
The situation in Arakelov geometry seems widely different as the arithmetic numerical invariants are no longer discrete, see, for example, the definition of $h^0_{Ar}$ in \cite{soulé_abramovich_burnol_kramer_1992}, which is used as the dimension of the space of Arakelov sections, the definitions of $h^0_{\theta}$ and $h^1_{\theta}$ in \cite{bost} and the definition of the arithmetic degree. This may lead to believe that the conservation of number principle is not relevant in this context. However, in this article, we will introduce an Arakelov theoretic version of the deformation to the normal cone and prove the conservation of some arithmetic Hilbert invariants along the deformation. The implicit aim of this paper is to develop a tool for more transparent proofs and constructions in arithmetic intersection theory, throughout a closer analogy with algebraic geometry. In a coming article, we will give a proof by deformation of the arithmetic Hilbert-Samuel theorem.\\

\subsection{Geometry of the tube and arithmetic Grauert tubes.}
In \cite{grauert}, Grauert initiated in a new paradigm concerning the notion of ampleness of a line. Grauert's ampleness criterion states that the ampleness of a line bundle is characterised by the fact that the total space of $\mcL^{\vee}$ denoted $\VV_X(\mcL)$ is a modification of a affine scheme, see in particular the version given in EGA II \cite{grothendieck}.\\
In Arakelov geometry, this shift from the properties of the line bundle to the properties of the total space $\VV_X(\mcL)$ also claims interesting results. \\
Indeed, considering the total space of an ample Hermitian line bundle $\overline{\mcL}$ over a scheme $X$ projective over $\ZZ$ allows to obtain geometric insight on the metric on $\mcL$. This is done by considering the corresponding \emph{analytic tube}.\\ As $\mcL$ is ample, the analytic space $$\VV_{X(\CC)}(\mcL) = \{(x,\varphi)\;|\; x\in X(\CC), \varphi \in \mcL^\vee_{x}\} $$ is mod-Stein, and the analytic tube associated to the metric on $\mcL$ is then: $$T = \{(x,\varphi)\;|\; |\varphi|_x\leq 1 \}$$ It is a holomorphically convex compact of $\VV_{X(\CC)}(\mcL)$.\\ This allows \cite{bostDwork} and \cite{randriam} to derive some lifting results, similar to the Hörmander's $L^2$ estimates  (see for example \cite{Manivel1993}) but without any smoothness assumption on $X$. Furthermore, in the formalism of $A$-schemes developed in \cite{Bost-Charles}, an arithmetic Grauert's ampleness criterion is given.\\

This article arises in this context where the hypothesis of smoothness is not needed, and where we work with the canonical topological structure induced by the morphism:
$$H^0(\VV_X(\mcL), \mcO_\VV) \to H^0(T,\mcO_\VV) $$
on the total space of sections: 
$$H^0(\VV_X(\mcL), \mcO_\VV) = \bigoplus_{n\in \NN} H^0(X,\mcL^{\otimes n})$$
 In \cite{Bost-Charles}, the semipositivity of a seminormed line bundle $\overline{\mcL}$ is then characterized by the properties of the topological spaces $H^0(\VV_X(\overline{\mcL}), p_X^*\mcF)$ where $p_X^*\mcF$ is the pullback to $\VV_X(\mcL)$ of a coherent sheaf $\mcF$ on $X$.\\
 
 Consequently, the natural framework for this article will be (possibly non-reduced) projective schemes over $\ZZ$ and upper semicontinuous seminorms on line bundles. \\

\subsection{Seminormed vector bundles and the semipositivity condition.}
As suggested in the previous paragraph, the framework used here allows a generalisation of the notion of Hermitian line. This is the notion of upper semicontinuous seminorms on line bundle.
Allowing upper semicontinuous seminorms and no stronger hypothesis of regularity, enables us to freely use \emph{the holomorphic convexity of the analytic tube as the definition of the semipositivity of a seminormed line bundle}. This terminology is coherent with the one given in \cite{zhang} for example, and it generalises the notion of semipositivity to seminormed line bundle. In particular, it allows us to take the holomorphic convex hull of analytic tubes to define semipositive seminormed structures on line bundles.\\ Although not needed here thanks to our definitions, as the holomorphically convex hull is an object arising from potential theory, it is to be noted that some regularity results may exist. For example, in \cite{Bedford-Taylor}, it is stressed that solutions to the Monge-Ampère equations are given as the boundary of some holomorphically convex hull of some compact in a space of bigger dimension that the initial problem. And furthermore, in the same paper, several regularity results are proven for the solution to the Monge-Ampère equation, which proves the regularity of the boundary of the holomorphically convex hull under certain conditions. Also, as a corollary of \cite{berman2007bergman}, we obtain some regularity results for holomorphic convex hulls, as used in the end of the proof of proposition \ref{passageAlenveloppeholo}.

\subsection{Arithmetic Hilbert invariants.} In Arakelov geometry, given a seminormed line bundle $\overline{\mcL}$ on a scheme $X$ projective $\ZZ$, topological structures on the space of sections $H^0(X,\mcL^{n})$ are given by the $L^p$-norms, with $1\leq p \leq \infty$. The more usual ones being the $L^2$-norm and the $L^\infty$-norm.\\
One important asset of using the total space of $\mcL$, is that, if $\mcF$ is a coherent sheaf, considering the analytic tube in $\VV_X(\mcL)$ endows the space $H^0(\VV_X(\mcL),p_X^*\mcF)=\bigoplus_{n\in \NN} H^0(X,\mcL^{n}\otimes \mcF)$ with a \emph{canonical} topological structure. Furthermore, under the condition that $\overline{\mcL}$ is a Hermitian line bundle, that the underlying space $X$ is reduced and that $\mcF$ is a Hermitian vector bundle, this topological structure can be understood via the $L^p$-norms, see \cite{Bost-Charles}.\\
This topological structure will allow us to define numerical invariants $\overline{c}_r(X,\overline{\mcL}, \mcF)$ and $\underline{c}_r(X,\overline{\mcL},\mcF)$ in $\RR\cup\{ +\infty \}$, inspired by the arithmetic intersection theory and the arithmetic Hilbert-Samuel theorem, that verify several important properties : 
\begin{itemize}
    \item It satisfies a projection formula for closed immersions.
    \item It is additive with respect to exact sequences of coherent sheaves.
    \item Assuming the results of arithmetic intersection theory: if $\overline{\mcL}$ is a Hermitian line bundle on a projective scheme $X$ over $\Spec \ZZ$ of dimension $d$ with smooth generic fiber, then:
        $$\overline{c}_d(X,\overline{\mcL}, \mcO_X)= \underline{c}_d(X,\overline{\mcL},\mcO_X)= \widehat{c_1}(\overline{\mcL})^d $$
\end{itemize}

Extracting numerical invariants from infinite dimensional lattices in topological spaces is to be put in perspective and in lineage with the generalisation of the notion of dimension of a complex analytic affine space given in \cite{Gelfand} with the definition of the functional dimension of linear topological spaces. \\


\subsection{Deformation to the projective completion of the cone and the conservation of number theorem.}
We will introduce the deformation to the projective completion of the cone over $A$, which is either a field or $\ZZ$. This is a projective variation on the deformation to the normal cone, containing the deformation to the normal cone as an open subscheme. \\
Consider the following data:
\begin{itemize}
    \item  $s: X \to \Spec A$ a projective scheme over $A$,
    \item a line bundle $\mcL$ over $X$ very ample over $A$,
    \item $Y$ a hyperplane section with respect to $\mcL$.
\end{itemize}

To this data, we will associate a pair $(\pi : D_YX\to \PP^1_A,\mcL')$, called \emph{the deformation to the projective completion of the cone}, where:
\begin{itemize}
    \item $D_YX$ is projective over $\PP^1_A$ and over $A$,
    \item the line bundle $\mcL'$ over $D_YX$ is ample over $A$,
    \item the restriction of $(D_YX,\mcL')$ over $1 \in \PP^1_\ZZ$ is $(X,\mcL)$,
\end{itemize}

We will also describe an action of $\GG_m$ on $(\pi : D_YX\to \PP^1_A,\mcL')$, called the \emph{natural transverse action of $\GG_m$}. It is an equivariant action of $\GG_m$ on $\VV_{D_YX}(\mcL') \to D_YX \to \PP^1_A$ where the action of $\GG_m$ on $\PP^1_A$ is the natural one. \\

Finally, if $A=\ZZ$ and if $\mcL$ is endowed with a seminormed structure. We will endow $\mcL'$ with a structure of semipositive seminormed line bundle.\\
The deformation to the projective completion of the cone associated to $(X,Y,\overline{\mcL})$ is then denoted $(D_YX, \overline{\mcL'})$, it verifies the following properties:
\begin{itemize}
    \item $\overline{\mcL'}$ is semipositive.
    \item the fiber over $1$ of $(D_YX, \overline{\mcL'})$ is $(X,\overline{\mcL})$, if $\overline{\mcL}$ is semipositive.
\end{itemize}

Then, if $T'$ is the analytic tube associated to $\overline{\mcL'}$, the action of $\GG_m$ on $\VV_{D_YX}(\mcL')$ will restrict to an action of $U(1)$ on $T'$.\\

Eventually, we will prove the following conservation of number result:

\begin{thm}\label{invarianceArithmeticHilbert}Let $X$ be a reduced projective scheme over $\ZZ$. Let $\overline{\mcL}$ be a semipositive seminormed line bundle over $X$, very ample over $\ZZ$. Assume furthermore that the metric is uniformly definite on $\overline{\mcL}$. Let $i:Y\to X$ be a hyperplane section with respect to $\mcL$.\\
Let $(D_YX, \overline{\mcL'})$ be the deformation to the projective completion of the cone associated to $(X,Y,\overline{\mcL})$.\\
Then $X$ and $D_YX_{|\infty}$ have the same dimension $d$.\\
Furthermore, $(X, \overline{\mcL}, \mcO_X)$ and $(D_YX_{|\infty}, \overline{\mcL'},\mcO_{D_YX_{|\infty}} )$ have the same arithmetic Hilbert invariants.\\
\end{thm}

\subsection{Strategy of proof.} The natural transverse action of $\GG_m$ will allow us to link sections of $\mcL^{\otimes n}$ on the copy of $X$ over $1$ to sections $\mcL_{|\infty}^{\otimes n}$ of $D_YX_{|\infty}$. It is done by extending sections on the fiber above $1$ to the whole space of the deformation $D_YX$ by using the action of $\GG_m$, and, then, by restricting those sections on $D_YX$ to fiber above $\infty$, see the proposition \ref{decompSelonActionGm} on the isotypic components. The compatibility of the natural transverse action of $\GG_m$ with the analytic tube of $\overline{\mcL'}$ will then ensure a corresponding relationship between the norms defining the respective topologies.
The rest of the proof is then a careful unraveling of the definition of the arithmetic Hilbert invariants and the use of the additivity of the arithmetic degree with respect to exact sequences of Hermitian vector bundles over $\Spec \ZZ$.

\subsection{Outline of the paper.}
The formalism of the geometry of the analytic tube will be introduced in section 2 including the extension of the notion of semipositivity to seminormed line bundles. In section 3, we will describe the canonical topological structure on the total space of sections and prove that this topology is an invariant of the equilibrium seminorm associated to the given seminorm. In section 4, we will introduce the Hilbert invariants and prove their main properties.
In section 5, we will give the definition of the deformation to the projective completion of the cone, we will state its main properties, including an exhaustive description of the natural transverse action of $\GG_m$. 
In section 6, we will describe the deformation of analytic tubes. Finally, in section 7, we will prove the conservation of the arithmetic Hilbert invariants along the deformation to the projective completion of the cone.\\
In the appendix, we will depict the relationship between the deformation to the projective completion of the cone and the deformation to the normal cone.

\subsection{Acknowledgements.} I would like to thank François Charles for sharing its views on Arakelov geometry with me, for the numerous discussions and for his perseverance in insisting on a deeper understanding of the construction developed in this article. This paper owes him a great intellectual debt.\\
This project has received funding from the European Research Council (ERC) under the European Union’s Horizon 2020 research and innovation programme (grant agreement No 715747).

\section{Seminormed vector bundles and the associated analytic tubes}

We first give a few definitions in the usual setting of Arakelov geometry and then give their counterpart in the framework of \cite{Bost-Charles}.
\subsection{Seminormed vector bundles\\}

\begin{defn} Let $X$ be a projective scheme over $\ZZ$. 
A \emph{seminorm} $|\cdot|$ on a locally free sheaf of finite rank $\mcF$ on $X$, is, for each reduced point $x$ of $ X_{\CC}$, a seminorm $|\cdot|_x$ on the complex vector space $\mcF_{|x} = \mcF_x/\mfm_x\mcF_x$.\\
Furthermore, the seminorm $|\cdot|$ is said to be \emph{upper semicontinuous} (resp. \emph{continuous}) if for every open set $\mcU$ of the analytic space $X_{\CC}$, and every $s\in H^0(\mcU,\mcF)$,  the map 
$$\begin{array}{ccccc}
|s| & : & \mcU & \to & \RR^{+} \\
 & & x & \mapsto & |s(x)|_x \\
\end{array}$$
is upper semicontinuous (resp. continuous).\\
A \emph{seminormed vector bundle} $\overline{\mcF} = (\mcF,|\cdot|)$ on $X$, is a locally free sheaf of finite rank $\mcF$ endowed with an upper semicontinuous seminorm invariant by complex conjugation.\\
\end{defn}

\begin{rem} The complex conjugation sends a complex point $x:\Spec \CC \to X$ to $\overline{x}:\Spec \CC \xrightarrow{\overline{\cdot}}\Spec \CC \to X$.\\
So, if $x\in X(\CC)$ is characterized by a point $p_x\in X$ and a morphism $\phi_x : \mcO_{X,p_x} \to \CC$, then $\overline{x}$ is characterized by $p_{\overline{x}} = p_x$ and $\phi_{\overline{x}} = \overline{\phi_x}$.\\
The complex conjugation then acts on $\mcF_{|x}=\mcF_{p_x} \otimes_{\mcO_{X,p_x}}\CC$ by defining a morphism: 
$$\begin{array}{ccccc}
\overline{(\cdot)} & : & \mcF_{|x} & \to & \mcF_{|\overline{x}} \\
 & & s\otimes \lambda & \mapsto & s\otimes \overline{\lambda} \\
\end{array}$$
A seminorm $|\cdot|$ on $\mcF$ is invariant by conjugation if, for $f \in \mcF_{|x}$, we have  $$|f|_x = |\overline{f}|_{\overline{x}}$$
\end{rem}

\begin{defn} Let $X$ be a projective scheme over $\ZZ$. A seminorm $|\cdot|$ on a seminormed vector bundle $\overline{\mcF}$ is said to be \emph{definite} if all the seminorms $|\cdot|_x$ are definite. $\overline{\mcL}$ is then called a \emph{normed vector bundle} and $|\cdot|$ a \emph{norm on $\mcL$}.\\
In the special case where $|\cdot|$ is definite continuous and Hermitian,  $\overline{\mcF}$ is usually called a \emph{Hermitian vector bundle} and $|\cdot|$ a continuous Hermitian metric.\\

\end{defn}

In the case of an upper semicontinuous seminorm on a vector bundle $\mcF$, we will need a more restrictive notion than the notion of definite seminorms:\\

\begin{defn}
Let $X$ be a projective scheme over $\ZZ$. A norm $|\cdot|$ on a normed vector bundle $\overline{\mcF}$ is said to be \emph{uniformly definite} if there exists a continuous norm $|\cdot|'$ on $\mcF$, such that $|\cdot|'\leq |\cdot|$.\\
\end{defn}

An example of a definite but non uniformly definite seminorm will be given below the proposition \ref{unfiormlydefinitetube}.\\

The analog of the notion of ample line bundle is given in \cite{zhang} and we also use the associated notion of semipositivity:

\begin{defn}\label{amplehermitian} Let $X$ be a projective scheme over $\ZZ$. A Hermitian line bundle $\overline{\mcL}$ is \emph{semipositive} if :
\begin{itemize}
    \item[i)] $\mcL$ is ample over $\Spec \ZZ$.
    \item[ii)] for every open set $\mcU$ of the analytic space $X_{\CC}$, and every non-vanishing section $s\in H^0(\mcU,\mcL)$,  the map 
$$\begin{array}{ccccc}
-\log|s|^2 & : & \mcU & \to & \RR\\
 & & x & \mapsto & -\log|s(x)|^2_x \\
\end{array}$$ is plurisubharmonic.
\end{itemize}
Furthermore, a Hermitian line bundle $\overline{\mcL}$ over $X$ is \emph{ample} if it is semipositive and for large enough $n$, there is a basis of $H^0(X, \mcL^{\otimes n})$, consisting of strictly effective sections, that is sections $s\in H^0(X, \mcL^{\otimes n})$ such that $|s|_x < 1$ for all $x \in X(\CC)$.
\end{defn}

This definition will be extended in definition \ref{definitionAmpleLineBundleHolomorphicallyConvex} to seminormed line bundles.\\

\begin{ex}
In the case of the projective space $\PP^l_\ZZ$, the line bundle $\mcO_\PP(1)$ can be endowed with the Fubini-Study metric $|\cdot|_{FS}$ :\\
Let $[x_0:\dots : x_l] \in \PP^l_\ZZ(\CC)$, then 
$\mcO_\PP(1)_{[x_0:\dots : x_l]}$ is isomorphic to $\Hom_\CC(\CC(x_0,\dots,x_l),\CC)$. 
If $P\in \Hom_\CC(\CC(x_0,\dots,x_l),\CC)$, we define $${|P|_{FS}}_{[x_0:\dots : x_l]} :=  \frac{|P(x_0,\dots,x_l)|}{(|x_0|^2+\dots + |x_l|^2)^{1/2}}$$
This defines a semipositive Hermitian line bundle $\overline{\mcO_\PP(1)} = (\mcO_\PP(1),|\cdot|_{FS} )$. \\
Furthermore, for all $\epsilon > 0$, $(\mcO_\PP(1),e^{-\epsilon}|\cdot|_{FS} )$ is an ample Hermitian line bundle.\\
\end{ex}

The following definition of pullback of seminormed vector bundle give further examples:\\

\begin{defn}
Let $f: X \to Y$ be a morphism between projective schemes over $\ZZ$. Then the \emph{pullback} $f^*\overline{\mcF}$ of a seminormed vector bundle $\overline{\mcF}= (\mcF, |\cdot|)$ over $Y$ is the seminormed vector bundle of underlying sheaf $f^*\mcF$ and such that the seminorm $|\cdot|'$ over $f^*\mcF$ verifies that the pullback map $\mcF_{|f(x)} \to (f^*\mcF)_{|x} $ is an isometry. 
\end{defn}

\begin{ex}
Let $X$ be a projective scheme over $\ZZ$, $\mcL$ be a line bundle over $X$ ample over $\ZZ$.
Now, let $\pi : X \to \PP^l_\ZZ$ be a closed immersion induced by global sections of $\mcL^{\otimes n}$, then $\pi^*(\mcO_\PP(1),|\cdot|_{FS})$ endows $\mcL^{\otimes n}$ with a structure of semipositive Hermitian line bundle. Taking the $n$-root of the seminorm, we get a structure of semipositive Hermitian line bundle on $\mcL$.\\
Furthermore, any semipositive Hermitian metric on an ample line bundle can be approximated uniformly by roots of pullbacks of Fubini-Study metric, see for example \cite{tian}, \cite{BoucheB1}, \cite{berman2007bergman}.\\

\end{ex}

\subsection{The analytic tube\\[2mm]}\label{analytictube}

Grauert proved that the ampleness of a line bundle $\mcL$ over $X$ a projective scheme over $\ZZ$ can be characterized by the fact that $\VV_X(\mcL)$ is a modification over $\ZZ$ of an affine scheme, where $\VV_X(\mcL)$ is the affine scheme over $X$ defined by the sheaf of algebra $\bigoplus_{n\in \NN} \mcL^{\otimes n}$, see EGA II \cite{grothendieck}[Chapter 2, Sections 8.8 to 8.10] .\\
In a similar manner, \cite{Bost-Charles} gives a definition of the \emph{analytic tube} associated to a seminormed line bundle $\overline{\mcL}$, which allows an corresponding definition of the semipositivity of a seminormed line bundle $\overline{\mcL}$ as a property of the analytic tube.\\

More precisely, the complex points of $\VV_X(\mcL)$ can be written as $$\VV_X(\mcL)(\CC) = \{(x,\varphi)\;|\; x\in X(\CC), \varphi \in \mcL^\vee_{|x}\}$$
as points of the total space of the dual line bundle $\mcL^\vee$.\\
If $|\cdot|$ is an upper semicontinuous seminorm on $\mcL$, then the dual pseudo-norm, also denoted $|\cdot|$, is lower semicontinuous and definite, thus $$T = \{(x,\varphi)\;|\; x\in X(\CC), \varphi \in \mcL^\vee_{|x}, |\varphi|_x\leq 1\}$$ is compact.
\begin{defn}
With the above notation, $T$ is called the \emph{analytic tube associated to $\overline{\mcL}$}.\\
\end{defn}

\begin{rem} With the above notation, the complex conjugation acts on $\VV_X(\mcL)(\CC)$ by sending $(x,\varphi)$ to $(\overline{x},\overline{\varphi})$ where $\overline{\varphi}(s) = \overline{\varphi(\overline{s})}$ for $s\in \mcL_{|x}$. Furthermore, the analytic tube associated to a seminormed line bundle is invariant by conjugation.\\ 
\end{rem}

\begin{prop}
The datum of $(\VV_X(\mcL),T)$ where $T$ is a compact subset of $\VV_X(\mcL)(\CC)$ containing the zero section $0_X \subset \VV_X(\mcL)(\CC)$, stable by multiplication on fibers over $X(\CC)$ by any complex number $\lambda$ such that $|\lambda| \leq 1$, and invariant by complex conjugation is equivalent to the datum of a seminormed line bundle.
\end{prop}
\begin{proof}
See \cite{Bost-Charles}.
\end{proof}

In particular, this proposition explains why the upper-semicontinuous condition arises in the definition of seminormed vector bundles. This definition will especially useful for the definition \ref{equilibriumseminorm} of the equilibrium seminorm associated to a ample seminormed line bundle.

Later on, we will use the terminology \emph{seminormed line bundle} for both notions.\\

\begin{ex}

The total space associated to $(\PP^N_\ZZ, \mcO_\PP(1))$ is the following: $$\VV_{\PP^N_{\ZZ}}(\mcO_\PP(1))(\CC) = \{([x_0:\dots : x_N],(y_0,\dots,y_N)) \;|\; [x_0:\dots : x_N]\in \PP^N(\CC),(y_0,\dots,y_N)\in \CC (x_0,\dots,x_N) \} $$
And then, the tube associated to the Fubiny-Study metric is
$$T = \{ ([x_0:\dots : x_N],(y_0,\dots,y_N)) \in \VV_{\PP^N_{\ZZ}}(\mcO_\PP(1))(\CC) \;|\; \sum |y_i|^2 \leq 1 \}  $$
\end{ex}

The notion of uniformly definite admits a convenient translation in this setting: 

\begin{prop}\label{unfiormlydefinitetube}
The seminormed line bundle $\overline{\mcL} = (\VV_X(\mcL),T)$ is uniformly definite if, and only if, $T$ contains a neighborhood of the zero section $0_X$.

\end{prop}

\begin{ex}
This setting makes it convenient to exhibit an example of definite seminorm which is not uniformly definite. For example,
$$T = \{ ([x_0: x_1],(y_0,y_1)) \;|\;  |y_0|^2+|y_1|^2 \leq  \frac{|x_0|}{|x_0|+|x_1|}\} \cup \{ ([0: 1],(0,y_1)) \;|\;  |y_1|\leq 1\} $$
defines a compact in $\VV_{\PP^1_{\ZZ}}(\mcO_\PP(1))(\CC)$ which is not a neighborhood of the zero section but which is not reduced to $([x_0: x_1],(0,0))$ on the fiber above each point $[x_0: x_1]$ of $\PP^1_\ZZ(\CC)$. Therefore, it defines a definite but non uniformly definite seminorm on the line bundle $\mcO(1)$ on $\PP^1_\ZZ$. 
\end{ex}

In this context, we can now extend the definition of a semipositive Hermitian line bundle.

\begin{defn}\label{definitionAmpleLineBundleHolomorphicallyConvex}
Let $X$ be a projective scheme over $\ZZ$. A seminormed line bundle $\overline{\mcL} = (\VV_X(\mcL),T)$ over $X$ is said to be \emph{semipositive} if, $\VV_X(\mcL)$ is a modification over $\ZZ$ of an affine scheme and if $T$ is a holomorphically convex compact in $\VV_X(\mcL)(\CC)$. \\
Furthermore, a seminormed line bundle $\overline{\mcL}$ over $X$ is \emph{ample} if it is semipositive and for large enough $n$, there is a basis of $H^0(X, \mcL^{\otimes n})$, consisting of strictly effective sections.

\end{defn}

To see that this extends definition \ref{amplehermitian}, see \cite{Bost-Charles}, EGA II \cite{grothendieck}[Chapter 2, Sections 8.8 to 8.10] and \cite{forstneric}[Chapter 2]. In particular, see the following remark \ref{petitejustificationdequivalence}.\\

A seminormed line bundle such that the underlying line bundle is ample has an associated structure of semipositive seminormed line bundle. 
Extending Berman's terminology (see \cite{berman2007bergman}) to this formalism, we have the following definition:

\begin{defn}\label{equilibriumseminorm}
Let $X$ be a projective scheme over $\ZZ$. Let $\overline{\mcL} = (\VV_X(\mcL),T)$ be a seminormed line bundle over $X$. Assume furthermore that the underlying line bundle $\mcL$ is ample. The \emph{equilibrium seminorm associated to $\overline{\mcL}$} is the seminorm defined by the holomorphically convex hull $\widehat{T}$ of $T$ in $\VV_X(L)(\CC)$.\\
\end{defn}

\begin{rem} The following points justify why the holomorphically convex hull of an analytic tube is still an analytic tube: \begin{itemize}
    \item Grauert's ampleness criterion states that if a holomorphic line bundle is ample then the corresponding total space is mod-Stein. Thus $\VV_X(\mcL)(\CC)$ is mod-Stein and the holomorphically convex hull of a compact is compact.
    \item If $f\in H^0(\VV_X(\mcL)(\CC),\mcO_\VV)=  H^0(\VV_X(\mcL),\mcO_\VV)\otimes \CC $, the action of the complex conjugation on the coefficients defines $\overline{f}\in H^0(\VV_X(\mcL)(\CC),\mcO_\VV)$. It verifies that $\overline{f}(\overline{(x,\varphi)}) = \overline{f((x,\varphi))}$ for $(x,\varphi)\in \VV_X(\mcL)(\CC)$. This justifies that taking the holomorphically convex hull of an invariant by conjugation compact is still invariant by conjugation.\\
\end{itemize}
\end{rem}

\begin{rem}\label{petitejustificationdequivalence}The equilibrium metric associated to a smooth metric $\phi$ on a smooth projective space defined in  \cite{berman2007bergman} is $\phi_e(x) = \sup \Tilde{\phi}(x) $ where the supremum is taken over all continuous metrics $\Tilde{\phi}\leq \phi$ with positive curvature. Denoting the associated tubes by $T$, $T_{\phi_e}$ and $T_{\Tilde{\phi}}$, we have $T_{\phi_e} = \cap T_{\Tilde{\phi}}$. Furthermore, the compacts $T_{\Tilde{\phi}}$ are holomorphically convex and contain $T$, see \cite{forstneric}[Chapter 2]. Thus, $\widehat{T} \subset T_{\phi_e}$.\\
Now, by \cite{Bost-Charles}, the seminorm associated to $\widehat{T}$ can be approximated by above by continuous norms $|\cdot|_i$ which verifies the plurisubharmonicity condition ii) of definition \ref{amplehermitian}. Hence, $\widehat{T} = \cap T_i \supset T_{\phi_e}$.\\
Finally, $\widehat{T}= T_{\phi_e}$. Hence, under the assumption of smoothness considered in \cite{berman2007bergman}, the terminology is consistent with Berman's terminology.
\end{rem}

\section{Bornological structure on the total space of sections}

\subsection{Seminorms on the space of sections of a seminormed vector bundle\\[2mm]}

Let $X$ be a projective scheme over $\ZZ$ and $\overline{\mcF}=(\mcF, |\cdot|)$ be a seminormed vector bundle over $X$.\\
Then, the space of sections $H^0(X_\CC,\mcF_{\CC})$ can be endowed with different seminorms.\\[2mm]
For example, if $X$ is a projective scheme over $\ZZ$ such that $X_{\CC}$ is reduced, we can define a seminorm on $H^0(X,\mcF)\otimes_\ZZ \CC$, by setting $$||s||_{\infty} =  \sup_{x\in X(\CC)} |s(x)|_x$$ for $s\in H^0(X,\mcF)\otimes_\ZZ \CC$.\\
In another hand, if, furthermore, we assume $X_{\CC}$ to be smooth, then a Riemannian continuous metric defines a measure $d\mu$ on the complex manifold $X_\CC$, and we can define a seminorm $$||s||_{L^2}^2 = \int_{X_\CC}|s(x)|_x^2 d\mu$$ for $s\in H^0(X,\mcF)\otimes_\ZZ \CC$.\\
If, furthermore, the Riemannian metric is assumed to be invariant by conjugation and $|\cdot|$ is a hermitian seminorm (e.g. $\overline{\mcF}$ is a seminormed line bundle), then $||\cdot||_{L^2}$ defines a Hilbertian seminorm on $ H^0(X,\mcF)\otimes_\ZZ \RR$.\\ 

It is to be noted that some theorems concerning this space of sections admit equivalent statements with respect to these two norms, see for example the arithmetic Hilbert-Samuel theorem in \cite{Abbes1995}.

\subsection{Bornological structure\\[2mm]} 
The formalism developed in \cite{Bost-Charles} defines a bornological structure on the algebra $H^0(\VV_X(\mcL), \mcO_\VV) = \bigoplus_n H^0(X,\mcL^{\otimes n})$ associated with $(\VV_X(\mcL), T)$, denoted $H^0(\VV_X(\mcL), T, \mcO_\VV)$ or $H^0(\VV_X(\overline{\mcL}), \mcO_\VV)$, and unifying the different seminorms mentioned above.\\

Given a complex vector space $V$ and a decreasing family of seminorms $(||\cdot||_i)_{i\in I}$ on $V$ , we may define a bornology $\mcB$ on $V$ by defining bounded sets to be the subsets $B$ of $V$ such that there exists $i\in I$ and $R>0$ with
$$\forall v \in B ,||v||_i \leq R$$
In other words, elements of $\mcB$ are the subsets of $V$ that are bounded with respect to one of the
seminorms $||\cdot||_i$, $i \in I$.\\

\begin{defn} Let $V$ be a complex vector space. A \emph{bornology of Hilbertian type on $V$} is a bornology $\mcB$ induced by a decreasing family of Hilbertian seminorms $\mcN = (||\cdot||_i)_{i\in \NN}$ such that the kernel of the $||\cdot||_i$ is eventually constant.\\
\end{defn}

\begin{rem}
Two families of seminorms $(||\cdot||_i)_{i\in \NN}$ and $(||\cdot||'_j)_{j\in \NN}$ defines the same bornology if:
\begin{itemize}
    \item for $i\in \NN$, there is $j\in \NN$ and $C>0$ such that $||\cdot||'_j \leq C||\cdot||_i$,
    \item for $j\in \NN$, there is $i\in \NN$ and $C>0$ such that $||\cdot||_i \leq C||\cdot||'_j$.\\
\end{itemize}
\end{rem}

\begin{ex}
Let $Y$ be a complex analytic space. Let $K$ be a compact in $Y$. Let $\mcF$ be coherent sheaf on $Y$. Then, $H^0(K,\mcF)$ is defined to be the topological vector space that is the locally convex direct limit of the spaces $H^0(U,\mcF)$, where $U$ runs through the open neighborhoods of $K$ in $Y$, endowed with their natural Fréchet space topology (see \cite{GunningRossi}[Chap. 8]).\\
By pulling back the bornology through the restriction map $H^0(Y,\mcF) \to H^0(K,\mcF)$, the complex vector space $H^0(Y,\mcF)$ is then endowed with a bornology of Hilbertian type, see \cite{Bost-Charles}.\\
\end{ex} 

\begin{defn}
Let $Y$ be a complex analytic space. Let $K$ be a compact in $Y$. Let $\mcF$ be coherent sheaf on $Y$. The complex space $H^0(Y,\mcF)$ endowed with the bornology described above is denoted $H^0(Y,K,\mcF)$.
\end{defn}

\begin{defn}
An \emph{arithmetic Hilbertian $\mcO_{\Spec \ZZ}$-module} $\overline{M}$ is a countably generated $\ZZ$-module $M$, such that $M_{\CC}=M\otimes_\ZZ \CC$ is endowed with a bornology of Hilbertian type, invariant by complex conjugation.\\
Furthermore we say that $\overline{M}$ is a \emph{graded Hilbertian $\mcO_{\Spec \ZZ}$-module}, if $M$ is graded.\\
\end{defn}

\begin{rem}
The invariant by conjugation property ensures that the bornology can be defined by using a decreasing family of Hilbertian seminorms invariant by conjugation.
\end{rem}

Let $\mcF$ be coherent sheaf on a separated scheme $Y$ of finite type over $\ZZ$. Let $K$ be an invariant by complex conjugation compact in the complex analytic space $Y(\CC)$.\\
Then $H^0(Y, \mcF)$ is endowed with a structure of arithmetic Hilbertian $\mcO_{\Spec \ZZ}$-module, where the bornology on $H^0(Y, \mcF)\otimes_\ZZ \CC = H^0(Y(\CC), \mcF_{|\CC})$ is given by the compact $K$ as above.\\

\begin{defn}
Let $Y$ be a separated scheme of finite type over $\ZZ$. Let $K$ be an invariant by complex conjugation compact in the complex analytic space $Y(\CC)$. Let $\mcF$ be a coherent sheaf on $Y$.\\
The arithmetic Hilbertian $\mcO_{\Spec \ZZ}$-module structure on $H^0(Y, \mcF)$ described above is denoted $H^0(Y,K,\mcF)$.\\
In the case of a seminormed line bundle $\overline{\mcL}= (\VV_X(\mcL),T)$, the graded arithmetic Hilbertian $\mcO_{\Spec \ZZ}$-module $H^0(\VV_X(\mcL),T,p_X^*\mcF)$ is also denoted $H^0(\VV_X(\overline{\mcL}),\mcF)$.\\
\end{defn}

\begin{rem} The invariance by conjugation of the bornology associated to $H^0(Y,K,\mcF)$ arises from the invariance by conjugation of the compact $K$, see \cite{Bost-Charles}.
\end{rem}

\subsection{Reduced and smooth cases\\[2mm]}

In the case of a reduced complex analytic space $Y$, $K$ a compact in $Y$ and $\mcF$ a coherent sheaf on $Y$. We have a more explicit description of the bornology associated to $H^0(Y, K,\mcF)$, which makes the relationship with the $L^2$-norms and $L^\infty$-norms more explicit.\\

\begin{prop}\label{reducedcasenomrlp}
 Let $Y$ be a reduced complex analytic space, let $K$ be a compact subset of $Y$, and let $\mcF$ be a Hermitian vector bundle over $Y$.\\
 Then the bornology on $H^0(K, \mcF)$ is induced by the limit of $L^\infty$-norms on relatively compact neighborhoods of $K$ in $Y$.\\
Furthermore, if we assume that $Y$ is smooth, then by choosing a Riemannian continuous metric on $Y$, we can define $L^p$-norms, for $1\leq p<\infty$. \\
Then the bornology on $H^0(K, \mcF)$ is induced by the limit of $L^p$-norms on relatively compact neighborhoods of $K$ in $Y$, for $1\leq p\leq \infty$.
\end{prop}

\begin{proof}
See \cite{Bost-Charles}.
\end{proof}

In the case of the total space of a line bundle, we have the following description
\begin{prop}\label{reducedcaseBornology}
Let $X$ be a smooth projective scheme over $\CC$, $\overline{\mcL}$ be a Hermitian line bundle over $X$, $\overline{\mcF}$ be a Hermitian vector bundle over $X$. Choose a Riemannian continuous metric on $X$.\\
Take $1 \leq p \leq \infty$.
Let $p_X: \VV_X(\mcL)\to X$ be the total space of $\mcL^{\vee}$.\\
Assume $(||\cdot||_j)_{j\in \NN}$ is a decreasing family of norms defining the bornology on $H^0(\VV_X(\overline{\mcL}),p_X^*\mcF)$. \\
Then, fix $j\in \NN$, there is $\epsilon>0$, $C>0$ such that, for all $n\in \NN$,
$$Ce^{n\epsilon} ||\cdot||_{L^p}\leq ||\cdot||_j \text{ , on }  H^0(X, \mcL^{\otimes n} \otimes \mcF)_\CC $$
Fix $\epsilon>0$, there is $j \in \NN$, $C>0$ such that, for all $n\in \NN$,
$$C||\cdot||_j \leq e^{n\epsilon} ||\cdot||_{L^p} \text{ , on }  H^0(X, \mcL^{\otimes n} \otimes \mcF)_\CC $$
\end{prop}

\begin{proof}
See \cite{Bost-Charles}.
\end{proof}

\begin{rem}
By definition of the bornology on $H^0(\VV_X(\overline{\mcL}), p_X^*\mcF)$, the bornology does not depend on any metric on $\mcF$. However, as the previous proposition suggests, one can use the data of a metric on $\mcF$ to define a family of decreasing norms which defines the bornology on $H^0(\VV_X(\overline{\mcL}), p_X^*\mcF)$.
\end{rem}

\subsection{Dependence of the bornology on the seminorm\\[2mm]}

Recall that we have the following definition :

\begin{defn}
Let $X$ be a projective scheme over $\ZZ$. Let $\overline{\mcL} = (\VV_X(\mcL),T)$ be a seminormed line bundle over $X$. Assume furthermore that the underlying line bundle $\mcL$ is ample. The \emph{associated equilibrium seminorm} is given by the holomorphically convex hull $\widehat{T}$ of $T$ in $\VV_X(L)(\CC)$.\\
\end{defn}
The following proposition proves that, on an ample line bundle, a seminorm and its associated equilibrium seminorm induce the same bornology on the complex vector space $H^0(\VV_X(\mcL)(\CC),\mcO_\VV)$.\\

\begin{prop}\label{passageAlenveloppeholo}
Let $X$ be a reduced projective space over $\ZZ$. Let $\overline{\mcL}= (\VV_X(\mcL),T)$ be a seminormed line bundle over $X$. Assume furthermore that $\mcL$ is ample. 
Then the restriction morphism $H^0(\VV_X(\mcL),\widehat{T},\mcO_\VV) \to H^0(\VV_X(\mcL),T,\mcO_\VV)$ is an isomorphism.
\end{prop}

\begin{proof}
Firstly, as $T \subset \widehat{T}$, this morphism is bounded, see \cite{Bost-Charles}.\\
As $X$ is reduced, the bornology on $H^0(\VV_X(\mcL),\widehat{T},\mcO_\VV)$ (resp. $H^0(\VV_X(\mcL),T,\mcO_\VV)$) is given by taking a basis of relatively compact neighborhoods of $\widehat{T}$ (resp. $T$) respectively and considering the sup norms on those neighborhoods.\\
As the seminorm on $\overline{\mcL}$ is upper semicontinuous, by Baire's theorem it can be strictly approximated by above by continuous norms on $\mcL$ (see Yosida regularization in \cite{Giaquinta}, for example).  Then, it can also be strictly approximated by smooth norms, i.e. smooth when pulled back in any resolution of singularities of $X(\CC)$. Hence, we have a basis of compact neighborhoods $(T_i)_{i\in \NN}$ of $T$ which correspond to smooth norms on $\mcL$. \\
\emph{Assume that we proved that $\widehat{T_i}$ is a neighborhood of $\widehat{T}$ for every $i\in \NN$}.\\
Then, for every bounded set $\mcB \subset H^0(\VV_X(\mcL),T,\mcO_\VV)$, there is $i\in \NN$ and $C>0$ such that $\mcB \subset \{f\in H^0(\VV_X(\mcL),\mcO_\VV) \;|\; ||f||_{\infty, T_i}\leq C\}$. Hence, as $\mcB \subset \{f\in H^0(\VV_X(\mcL),\mcO_\VV) \;|\; ||f||_{\infty, \widehat{T_i}}\leq C\}$, and $\widehat{T_i}$ is a neighborhood of $\widehat{T}$. Then, $\mcB$ is bounded in $H^0(\VV_X(\mcL),\widehat{T},\mcO_\VV)$ and this concludes. \\

So, we are left to prove that for any compact neighborhood $T'$ of $T$, corresponding to a smooth norm on $\mcL$, $\widehat{T'}$ is a compact neighborhood of $\widehat{T}$.
This is true if $\widehat{T'}$ correspond to a continuous norm on $\mcL$. Indeed, by using the action of $\GG_m$ on the fibers of $\VV_X(\mcL)$, we get that there is a $\epsilon>0$ such that $T\subset e^{-\epsilon}\cdot T'$. Thus $\widehat{T} \subset \widehat{e^{-\epsilon}\cdot T'} = e^{-\epsilon}\cdot \widehat{T'}$. And as $\widehat{T'}$ corresponds to a continuous norm, $e^{-\epsilon}\cdot \widehat{T'}$ is in the interior of $\widehat{T'}$. Thus, $\widehat{T'}$ is a neighborhood of $\widehat{T}$.\\

Finally, we are left to prove that if $T$ correspond to a smooth norm $|\cdot|$ on $\mcL$, the associated equilibrium seminorm is continuous.\\
By the resolution of singularities theorem of  \cite{hironaka}, there is a proper morphism $f : M \to X(\CC)$, where $M$ is a smooth scheme over $\CC$. 
As $\widehat{f^{-1}(T)} = f^{-1}(\widehat{T})$, we can assume that we work over a manifold, with an ample line bundle endowed with a smooth norm. By \cite{Bost-Charles}, the norm associated to $\widehat{T}$ is the pointwise decreasing limit of continuous norm $|\cdot|_i$ which verifies the plurisubharmonicity condition ii) of definition \ref{amplehermitian}. Reciprocally, if $|\cdot|'\geq |\cdot|$ verifies the condition ii) of definition \ref{amplehermitian} then the associated tube is holomorphically convex, see \cite{forstneric}[Theorem 2.5.2]. Hence $\widehat{T}$ is the infimum on all norms $|\cdot|'\geq |\cdot|$ which verifies the condition ii) of definition \ref{amplehermitian}. In this case, Berman proved in \cite{berman2007bergman} that $\widehat{T}$ is a uniform limit of roots of pullbacks of Fubini-Study norms. \\
This concludes.\\
\end{proof}

\begin{rem} One may ask if in the case of a reduced mod-Stein complex analytic space $X$ (for a definition of a mod-Stein space, see \cite{Bost-Charles}), a compact $K$ containing the compact analytic subspace and its holomorphically convex hull $\widehat{K}$ induce the same bornology on the complex space $H^0(X, \mcO_X)$, i.e if the map of bornological complex vector space  $H^0(X,\widehat{K},\mcO_X)\to H^0(X,K,\mcO_X)$ is an isomorphism. That is true for mod-Stein manifolds and is a corollary of the following proposition:\\

\end{rem}

\begin{prop}
Let $X$ be a Stein manifold. Let $K$ be a compact set in the analytic space $X$. Then, the map $H^0(X,\widehat{K},\mcO_X)\to H^0(X,K,\mcO_X)$ is an isomorphism.
\end{prop}

\begin{proof} As in the proof of proposition \ref{passageAlenveloppeholo}, the main point is to prove that if $K'$ is a compact neighborhood of $K$, then $\widehat{K'}$ is a compact neighborhood of $\widehat{K}$.\\
For this purpose, we can assume that $K'$ is holomorphically convex. \\
In \cite{siu} and \cite{hirschowitz}, a distance $d_\Omega$ to the boundary of an open subset $\Omega\subset X$ is constructed, compatible to the restriction to an open subset. This distance verifies that if $\Omega$ is Stein, then $-log_{\Omega}$ is plurisubharmonic.\\ 
As $K'$ is holomorphically convex, $K'$ admit a basis of Runge Stein neighborhood, see \cite{forstneric}.\\
Let $\Omega$ be a Runge Stein neighborhood of $K'$ such the distance from any point of $\partial K'$ to $\partial\Omega$ is strictly less than the distance from $K$ to $\partial K'$. As $\Omega$ is Runge, the holomorphically convex hull with respect to holomorphic functions on $\Omega$ $\widehat{K}_\Omega$ is the same as the holomorphically convex hull with respect to holomorphic functions on $X$.\\
Then, as $-\log d_\Omega$ is plurisubharmonic, $\sup\limits_{\widehat{K}_\Omega} -\log d_\Omega = \sup\limits_K -\log d_\Omega $.\\
Thus, the distance from $\widehat{K}_\Omega = \widehat{K}$ to the boundary of $\Omega$ is the distance from $K$ to the boundary of $\Omega$ and is strictly bigger than the distance of any point of $\partial K'$ to the boundary of $\Omega$.\\
Thus $\widehat{K}$ is in the interior of $K'$. This concludes.\\
\end{proof}

\section{Arithmetic Hilbert invariants}

\subsection{Arithmetic degree for seminormed Hermitian coherent sheaves\\[2mm]}

We will use the usual notion of arithmetic degree to develop invariants for graded arithmetic Hilbertian $\mcO_{\Spec \ZZ}$-modules.

\begin{defn}
A \emph{seminormed Hermitian coherent sheaf} over $\Spec \ZZ$ is a pair $\overline{M} :=(M,||\cdot||)$ where $M$ is a module of finite type over $\ZZ$ and $||\cdot||$ is a Hilbertian seminorm over $M_{\RR} \colon= M\otimes_{\ZZ}\RR$.\\
The \emph{arithmetic degree} of $\overline{M}$ is defined by: 
\begin{align*}
    \widehat{\chi}(\overline{M}) &  := 0 & \text{ if } M=0,\\
    & := +\infty & \text{ else, if } ||\cdot|| \text{ is not a norm,}\\
    &:= -\log \text{covol}(\overline{M}) + \log \#M_{\text{tor}} & \text{ otherwise}
\end{align*}
where the covolume is computed using the unique translation invariant Radon measure $\mu_{\overline{M}}$ on $M_{\RR}$ that satisfies the following condition: for any orthonormal basis $(e_1 , \dots  , e_N )$ of $(M_\RR , ||\cdot|| )$,$$ \mu_{\overline{M}}\left(\sum [0,1]e_i \right) = 1$$

\end{defn}

\subsection*{Additivity with respect to exact sequences.}

One of the most important property of the arithmetic degree is the additivity with respect to exact sequences: \\

\begin{lemma}\label{additivityclassical} Let $0 \to \overline{N} \to \overline{M} \to \overline{P} \to 0$ be an exact sequence of seminormed Hermitian coherent sheaves over $\Spec \ZZ$, i.e. the seminorm on $P_\RR$ is the quotient seminorm and the seminorm $N_\RR$ is the induced seminorm. Then $$\widehat{\chi}(\overline{M}) = \widehat{\chi}(\overline{P}) + \widehat{\chi}(\overline{N})$$
\end{lemma}

\begin{proof}
See, for instance, \cite{chambertloir}.
\end{proof}

\begin{cor}\label{additivfiltration}
Let $\overline{M}= (M, ||\cdot ||)$ be a seminormed Hermitian coherent sheaf over $\Spec \ZZ$.
Let $\{0\} = F_0 \subset \dots \subset F_n = M$ be a filtration of $M$ by coherent sheaves over $\Spec \ZZ$. Let $\overline{F_i/F_{i+1}}$  be the sub-quotient of $\overline{M}$, defined by $\overline{F_i/F_{i+1}}= (F_i/F_{i+1}, ||\cdot||_{sq})$ . Then, we have $$\widehat{\chi}(\overline{M}) = \sum \widehat{\chi}(\overline{F_i/F_{i+1}})$$
\end{cor}

\begin{proof}
This is a repetitive application of lemma \ref{additivityclassical}.
\end{proof}

\subsection{Hilbert invariants for graded arithmetic Hilbertian $\mcO_{\Spec \ZZ}$-modules}

\subsubsection{Definitions.}
Inspired by the arithmetic Hilbert-Samuel theorem, we define the following numerical invariants:

\begin{defn}\label{invariantforgradedmodule}
Let $M_\bullet$ be a graded arithmetic Hilbertian $\mcO_{\Spec \ZZ}$-module.
Let $r \geq d$.\\
Let $(||\cdot ||_i)_{i\in \mathbb{N}}$ be a decreasing family of Hilbertian seminorms invariant by complex conjugation defining the bornology on $M_\bullet$.
Let  $$\overline{c}_r(M_\bullet, (||\cdot||_i)) = \lim_{i}\limsup_n \frac{r!}{n^{r}}\widehat{\chi}(M_n, ||.||_i)$$
$$\underline{c}_r(M_\bullet, (||\cdot||_i)) = \lim_{i}\liminf_n \frac{r!}{n^{r}}\widehat{\chi}(M_n, ||.||_i)$$

\end{defn}

\begin{rem}
As the seminorms are supposed to be invariant by complex conjugation, the seminorms $(||\cdot||_i)$ defines Hilbertian seminorms on $M_n \otimes_\ZZ \RR$, so $\widehat{\chi}(M_n, ||\cdot||_i)$ is well-defined.\\
\end{rem}

\begin{rem}
$\underline{c}_r$ and $\overline{c}_r$ are well-defined since $||\cdot||_i  \geq ||\cdot||_{i+1}$ implies that  $\widehat{\chi}(M_n, ||\cdot||_i) \leq \widehat{\chi}(M_n, ||\cdot||_{i+1})$.\\
\end{rem}

\begin{rem}
We have $\overline{c}_r(M_\bullet, (||\cdot||_i)) \geq \underline{c}_r(M_\bullet, (||\cdot||_i))$.
\end{rem}

Let us prove that this notion does not depend on the choice of the family of seminorms $(||.||_i)_{i\in I}$.\\

\begin{lemma}
Let $A_\bullet$ be a graded algebra over $\ZZ$, finitely generated in degree $1$ of Krull dimension $d$. Let $M_\bullet$ be a graded arithmetic Hilbertian $\mcO_{\Spec \ZZ}$-module and assume furthermore that $M_\bullet$ is a finite graded $A_\bullet$-module.
Let $r \geq d$.\\
Let $(||\cdot||_i)_{i\in \NN}$, $(||\cdot||'_j)_{j\in \NN}$ be two decreasing families of Hilbertian seminorms invariant by complex conjugation both defining the bornology on $M_\bullet$.
Then, we have $$\overline{c}_r(M_\bullet, (||\cdot||_i)) = \overline{c}_r(M_\bullet, (||\cdot||'_j))$$
$$\underline{c}_r(M_\bullet, (||\cdot||_i)) = \underline{c}_r(M_\bullet, (||\cdot||'_j))$$
\end{lemma}

\begin{proof}
By the geometric Hilbert-Samuel theorem, the rank $r(n)$ of the free part of $M_n$ is a polynomial in $n$ of degree $\leq d-1$ for $n\gg 0$. \\
Now, fix $i \in \NN$, there is $j_0 \in \NN$ and $C>0$ such that, for $j\geq j_0$ $||\cdot||_i \geq C||\cdot||'_{j}$.\\
Hence, if $M_n$ is non zero, $\widehat{\chi}(M_n, ||\cdot||_i)\leq \widehat{\chi}(M_n, C||\cdot||'_{j})$.\\
That is $\widehat{\chi}(M_n, ||\cdot||_i)\leq \widehat{\chi}(M_n, ||\cdot||'_{j}) -r(n)\log C$.
Dividing by $n^r$, taking the limits on $n\to \infty$, then $j\to \infty$, then $i\to \infty$ get us the inequalities : $$\overline{c}_r(M_\bullet, (||\cdot||_i)) \leq \overline{c}_r(M_\bullet, (||\cdot||'_j))$$
$$\underline{c}_r(M_\bullet, (||\cdot||_i)) \leq \underline{c}_r(M_\bullet, (||\cdot||'_j))$$
By symmetry, this concludes.\end{proof}

\begin{notn} As $\underline{c}_r$ and $\overline{c}_r$ do not depend on the choice of the family of seminorms but only on the bornology $\mcB$ on $M_\bullet$, we note $\underline{c}_r(M_\bullet, \mcB)$ for $\underline{c}_r(M_\bullet, (||\cdot||_i))$ and $\overline{c}_r(M_\bullet, \mcB)$ for $\overline{c}_r(M_\bullet, (||\cdot||_i))$.\\
\end{notn}

As explained in the previous paragraph, if $X$ is a projective scheme over $\Spec \ZZ$ of dimension $d$, if $\overline{\mcL}$ is a seminormed line bundle over $X$ which is ample over $\ZZ$, and if $\mcF$ is coherent sheaf over $X$ and $p_X:\VV_X(\mcL) \to X$ is the total space of $\mcL^{\vee}$, then $H^0(\VV_X(\mcL), p_X^*\mcF)$ inherits a structure of graded arithmetic Hilbertian $\mcO_{\Spec \ZZ}$-module, which is a finite module over the graded algebra $H^0(\VV_X(\mcL), \mcO_\VV)$ of dimension $d$.\\

\begin{defn}
With the above notation, if $\mcB$ is the bornology associated with the graded arithmetic Hilbertian $\mcO_{\Spec \ZZ}$-module $H^0(\VV_X(\overline{\mcL}), p_X^*\mcF)$ and $r\geq d$ then we denote by 
\begin{center}
    $\overline{c}_r(X,\overline{\mcL}, \mcF) $ and $\underline{c}_r(X,\overline{\mcL}, \mcF)$
\end{center} the invariants $\overline{c}_r(H^0(\VV_X(\mcL), p_X^*\mcF), \mcB) $ and $\underline{c}_r(H^0(\VV_X(\mcL), p_X^*\mcF), \mcB)$, respectively.\\
 $\overline{c}_r(X,\overline{\mcL}, \mcF) $ and $\underline{c}_r(X,\overline{\mcL}, \mcF)$ are called the \emph{arithmetic Hilbert invariants associated to  $(X,\overline{\mcL},\mcF)$}.\\
\end{defn}

\begin{rem} It is to be noted that, on the contrary to $\overline{\mcL}$, no archimedean data is needed on the coherent sheaf $\mcF$. However, if $\mcF$ is a vector bundle, endowing $\mcF$ with a Hermitian structure allows convenient descriptions of the arithmetic Hilbert invariants.

\end{rem}
\subsubsection{Other descriptions.}

In the case of a Hermitian line bundle on a reduced projective scheme, these invariants admit the following description:\\

\begin{prop}\label{casreduitContinueinvariants}Let $X$ be a reduced projective scheme over $\ZZ$ of dimension $d$, such that the generic fiber is smooth. Let $\overline{\mcL}$ be a Hermitian line bundle over $X$, ample over $\ZZ$. Let $\overline{\mcF}$ be a Hermitian vector bundle over $X$. Fix a Riemannian continuous metric on $X(\CC)$ invariant by complex conjugation, and denote $||\cdot||_{L^2}$ the associated $L^2$-norms on $H^0(X(\CC),\mcL^{\otimes n}\otimes \mcF)$.
Let $r\geq d$. \\
Then, $$\overline{c}_r(X,\overline{\mcL}, \mcF) = \limsup_n \frac{r!}{n^{r}}\widehat{\chi}(H^0(X,\mcL^{\otimes n}\otimes \mcF), ||.||_{L^2})$$
$$\underline{c}_r(X,\overline{\mcL}, \mcF) = \liminf_n \frac{r!}{n^{r}}\widehat{\chi}(H^0(X,\mcL^{\otimes n}\otimes \mcF), ||.||_{L^2})$$
In particular, using arithmetic intersection theory, we have $$\overline{c}_d(X,\overline{\mcL},\mcO_X) = \underline{c}_d(X,\overline{\mcL},\mcO_X) = \widehat{c_1}(\overline{\mcL})^d$$
\end{prop}

\begin{proof}
This is a direct application of proposition \ref{reducedcaseBornology}.
\end{proof}

Here, we use the notion of $h^0_\theta$ as developed in \cite{bost}. When the seminormed line bundle is ample, the arithmetic Hilbert invariants admit a description in terms of $h^0_\theta$:

\begin{prop}\label{lienh0theta}
Let $X$ be a projective scheme over $\Spec \ZZ$ of dimension $d$, $\overline{\mcL}$ be an ample seminormed line bundle over $X$. Let $\mcF$ be a coherent sheaf on $X$. Let $r\geq d$. Let $(||\cdot||_i)$ be a decreasing family of Hilbertian seminorms defining the bornology on $H^0(\VV_X(\mcL),p_X^*\mcF)$. Then 
$$\overline{c}_r(X,\overline{\mcL}, \mcF) = \lim_i \limsup_n  \frac{r!}{n^{r}}h^0_\theta(H^0(X,\mcL^{\otimes n}\otimes \mcF), ||.||_{i})$$ $$\underline{c}_r(X,\overline{\mcL}, \mcF) = \lim_i \liminf_n  \frac{r!}{n^{r}}h^0_\theta(H^0(X,\mcL^{\otimes n}\otimes \mcF), ||.||_{i})$$
\end{prop}

\begin{proof}
By the Poisson-Riemann-Roch formula (see \cite{bost}[2.2.2]),
$$h^0_\theta(H^0(X,\mcL^{\otimes n}\otimes \mcF), ||.||_{i})-h^1_\theta(H^0(X,\mcL^{\otimes n}\otimes \mcF), ||.||_{i}) = \widehat{\chi}(H^0(X,\mcL^{\otimes n}\otimes \mcF), ||.||_{i})$$
Then, as $\overline{\mcL}$ is ample, by \cite{Bost-Charles}, $H^0(\VV_X(\overline{\mcL}),p_X^*\mcF)$ is $h^1_\theta$-finite. That is, there is an $i_0$ such that for $i\geq i_0$, $(H^0(\VV_X(\mcL),p_X^*\mcF), ||.||_{i})$ is $h^1_\theta$-finite. \\
This implies that, for $i\geq i_0$, $h^1_\theta(H^0(X,\mcL^{\otimes n}\otimes \mcF), ||.||_{i})$ converges to $0$ when $n$ goes to $\infty$.\\
This concludes.
\end{proof}

In the ample case, the arithmetic Hilbert invariants also has a description in terms of $h^0_{Ar}$:

\begin{cor}
Let $X$ be a projective scheme over $\Spec \ZZ$ of dimension $d$, $\overline{\mcL}$ be an ample seminormed line bundle over $X$. Let $\mcF$ be a coherent sheaf on $X$. Let $r\geq d$. Let $(||\cdot||_i)$ be a decreasing family of Hilbertian seminorms defining the bornology on $H^0(\VV_X(\mcL),p_X^*\mcF)$. Then 
$$\overline{c}_r(X,\overline{\mcL}, \mcF) = \lim_i \limsup_n  \frac{r!}{n^{r}}h^0_{Ar}(H^0(X,\mcL^{\otimes n}\otimes \mcF), ||.||_{i})$$ $$\underline{c}_r(X,\overline{\mcL}, \mcF) = \lim_i \liminf_n  \frac{r!}{n^{r}}h^0_{Ar}(H^0(X,\mcL^{\otimes n}\otimes \mcF), ||.||_{i})$$
\end{cor}

\begin{proof}
This is a direct consequence of proposition \ref{lienh0theta} and of the comparison results between $h^0_\theta$ and $h^0_{Ar}$ presented in \cite{bost}[Chapter 3].\\\end{proof}

\subsubsection{Lower bounds.} As defined earlier, we may have the arithmetic Hilbert invariants equal to $-\infty$. However, we prove that this is not the case.\\
Firstly, let's prove the following lemma which will allow a reduction to the ample case :

\begin{lemma}\label{scaleArithmeticInvariants} Let $X$ be a projective scheme over $\Spec \ZZ$ of dimension $d$, $\overline{\mcL} = (\mcL, |\cdot|)$ be a seminormed line bundle, ample over $\Spec \ZZ$. Let $\mcF$ be a non-zero coherent sheaf over $X$. Let $\alpha \in \RR$. Consider the seminormed line bundle $\overline{\mcL'}= (\mcL, e^{\alpha}|\cdot|)$. Then,
$$\overline{c}_r(X,\overline{\mcL'}, \mcF)= \overline{c}_r(X,\overline{\mcL}, \mcF) -\alpha $$
$$\underline{c}_r(X,\overline{\mcL'}, \mcF)= \underline{c}_r(X,\overline{\mcL}, \mcF) -\alpha $$
\end{lemma}

\begin{proof}
Let $(||\cdot||_i)_{i\in \NN}$ be a decreasing family of Hilbertian seminorms defining the bornology on $H^0(\VV_X(\overline{\mcL}), p_X^*\mcF)$ orthogonal for the direct sum $\bigoplus_{n\in \NN} H^0(X,\mcL^{\otimes n} \otimes \mcF)$ and invariant under the complex conjugation. By \cite{Bost-Charles}, this bornology is the unique Hilbertian bornology $\mcB$ such that for all $(x,\varphi)$ in the tube $T$ associated to $|\cdot|$, the map $H^0(\VV_X(\mcL)(\CC), p_X^*\mcF) \to \mcF_x\otimes_{\mcO_{X,x}} \mcO_{\VV_{X}(\mcL)^\mathrm{an},(x,\varphi)}$ sending $\sum s_n$ to $\sum {s_n}_{(x,\varphi)}$ is bounded, where the topology is the natural sequence topology.\\
Now, consider the point $(x,e^{\alpha}\varphi)$. Let $s\in \mcL_x$ a trivialisation of $\mcL$ in a neighborhood of $x$.
Then the map from $\mcF_x\otimes_{\mcO_{X,x}} \mcO_{\VV_{X}(\mcL)^\mathrm{an},(x,\varphi)}$  to $ \mcF_x\otimes_{\mcO_{X,x}} \mcO_{\VV_{X}(\mcL)^\mathrm{an},(x,e^{\alpha}\varphi)}$ which restrict to the identity on $\mcF_x\otimes_{\mcO_{X,x}} \mcO_{X^{an},x}$ and that sends $t_x \otimes s$ to $t_x \otimes e^{-\alpha} s$ for $t_x\in \mcF_x$ is a strict isomorphism.\\
Then, let $(||\cdot||'_i)_{i\in \NN}$ be the decreasing family of Hilbertian seminorms on $H^0(\VV_X(\mcL)(\CC), p_X^*\mcF)$ defined by $||\sum s_n||'_i = || \sum e^{n\alpha }s_n||_i$.
Then, $(||\cdot||'_i)_{i\in \NN}$ define the unique Hilbertian bornology $\mcB'$ such that, for all $(x,e^{\alpha}\varphi)$ in the tube $e^\alpha T$ associated to $e^\alpha|\cdot|$, the map $H^0(\VV_X(\mcL)(\CC), p_X^*\mcF) \to \mcF_x\otimes_{\mcO_{X,x}} \mcO_{\VV_{X}(\mcL)^\mathrm{an},(x,\lambda\varphi)}$ sending $\sum s_n$ to $\sum {s_n}_{(x,e^{\alpha}\varphi)}$ is bounded.\\
Indeed, this last map is the composition of bounded maps:  $$(H^0(\VV_X(\mcL)(\CC), p_X^*\mcF),\mcB')\to (H^0(\VV_X(\mcL)(\CC), p_X^*\mcF),\mcB) \to \mcF_x\otimes_{\mcO_{X,x}} \mcO_{\VV_{X}(\mcL)^\mathrm{an},(x,\varphi)} \to \mcF_x\otimes_{\mcO_{X,x}} \mcO_{\VV_{X}(\mcL)^\mathrm{an},(x,e^{\alpha}\varphi)}$$ where the first map sends $\sum s_n$ to $\sum e^{\alpha n} s_n$.\\
Hence, the bornology on $H^0(\VV_X(\overline{\mcL'}), p_X^*\mcF)$ is given by the decreasing family of Hilbertian seminorms $(||\cdot||'_i)_{i\in \NN}$.\\
Then, the conclusion is a direct application of the definition of the arithmetic Hilbert invariants.\\
\end{proof}

In the ample case, we have the positivity of the arithmetic Hilbert invariants : 
\begin{cor}\label{potisiviteinvariantsFirst}
Let $X$ be a projective scheme over $\Spec \ZZ$ of dimension $d$, $\overline{\mcL}$ be an ample seminormed line bundle over $X$. Let $\mcF$ be a non-zero coherent sheaf over $X$. Let $r\geq d$.\\
Then, $$\overline{c}_r(X,\overline{\mcL}, \mcF)> 0$$ $$\underline{c}_r(X,\overline{\mcL}, \mcF) > 0$$
\end{cor}

\begin{proof} Let $\overline{\mcL} = (\mcL, |\cdot|)$. \\
Firstly, it is a consequence of proposition \ref{lienh0theta} and the fact that $h^0_\theta$ is non-negative that we have the non-negativity of the arithmetic Hilbert invariants in the ample case. Moreover, there is $\epsilon>0$ such that $(\mcL, e^{\epsilon}|\cdot|)$ is an ample seminormed line bundle, see \cite{charlesarxiv}[Proposition 2.5] for example. Hence, by lemma \ref{scaleArithmeticInvariants}, we have
$$\overline{c}_r(X,\overline{\mcL}, \mcF)\geq \epsilon $$ $$\underline{c}_r(X,\overline{\mcL}, \mcF) \geq \epsilon$$
This concludes.
\end{proof}

Then, comparison between seminorms implies the following comparison between arithmetic Hilbert invariants:

\begin{lemma}\label{monotonieMetric}
Let $X$ be a projective scheme over $\Spec \ZZ$ of dimension $d$, $\mcL$ is a line bundle over $X$ which is ample over $\ZZ$. Let $|\cdot| \geq |\cdot|'$ be two seminorms over $\mcL$. Let $\mcF$ be a coherent sheaf on $X$. Let $r\geq d$.\\
Then $$\overline{c}_r(X,(\mcL,|\cdot|), \mcF)\leq  \overline{c}_r(X,(\mcL,|\cdot|'), \mcF)$$  $$\underline{c}_r(X,(\mcL,|\cdot|), \mcF)\leq  \underline{c}_r(X,(\mcL,|\cdot|'), \mcF)$$
\end{lemma}

\begin{proof}
Let $T$, $T'$ be respectively the analytic tube corresponding to $|\cdot|$ and $|\cdot|'$. As $|\cdot| \geq |\cdot|'$, we have $T \supset T'$. Hence the map $H^0(\VV_X(\mcL),T,p_X^*\mcF)\to H^0(\VV_X(\mcL),T',p_X^*\mcF)$ is bounded.\\
This implies that if $(||\cdot||_i)_{i\in \NN}$, $(||\cdot||'_j)_{j\in \NN}$ are two families of norms defining the bornologies on $H^0(\VV_X(\mcL),T,p_X^*\mcF)$ and $ H^0(\VV_X(\mcL),T',p_X^*\mcF)$ respectively, then if $i\in \NN$, there exists $j\in \NN$, $C>0$ such that $||\cdot||_i\geq C ||\cdot||'_{j}$.\\
This concludes.\end{proof}

Finally, 

\begin{prop}\label{potisiviteinvariants}
Let $X$ be a projective scheme over $\Spec \ZZ$ of dimension $d$, $\overline{\mcL} = (\mcL, |\cdot|)$ be a seminormed line bundle over $X$, ample over $\ZZ$. Let $\mcF$ be a coherent sheaf over $X$. Let $r\geq d$.\\
Then, $$\overline{c}_r(X,\overline{\mcL}, \mcF)>-\infty$$ $$\underline{c}_r(X,\overline{\mcL}, \mcF) > -\infty$$
\end{prop}

\begin{proof} This corollary is a direct consequence of proposition \ref{potisiviteinvariantsFirst} and the fact that we can endow $\mcL$ with a Hermitian metric $|\cdot|'$, making $(\mcL, |\cdot|')$ an ample Hermitian line bundle over $X$. Indeed, then, there is $\alpha \in \RR$ such that $|\cdot| < e^{\alpha}|\cdot|'$, and thus by proposition \ref{monotonieMetric}, we can assume that $|\cdot| = e^{\alpha}|\cdot|'$.\\
Let $\overline{\mcL'} = (\mcL, |\cdot|')$. Then, to conclude, by lemma \ref{scaleArithmeticInvariants} and corollary \ref{potisiviteinvariantsFirst}, we have 
$$\overline{c}_r(X,\overline{\mcL}, \mcF)= \overline{c}_r(X,\overline{\mcL'}, \mcF) -\alpha >-\alpha$$
$$\underline{c}_r(X,\overline{\mcL}, \mcF)= \underline{c}_r(X,\overline{\mcL'}, \mcF) -\alpha >-\alpha$$
\end{proof}
\subsubsection{Projection formula.} The compatibility of the bornological structure with the pushforward of coherent sheaves yields the following proposition :

\begin{prop}\label{ArithmCompatibilitySubspaces}Let $X$ be a projective scheme over $\ZZ$ of dimension $d$. Let $\overline{\mcL}$ be a seminormed line bundle over $X$. Let $i : Y \to X$ be a closed immersion and $\mcF$ a coherent sheaf on $Y$. Let $r\geq d$\\
Then,
$$\overline{c}_r(X,\overline{\mcL}, i_*\mcF) = \overline{c}_r(Y,i^*\overline{\mcL}, \mcF)$$
$$\underline{c}_r(X,\overline{\mcL}, i_*\mcF) = \underline{c}_r(Y,i^*\overline{\mcL}, \mcF)$$

\end{prop}

\begin{proof}Let $T$ be the tube associated to the line bundle $\overline{\mcL}$.\\
Consider the following Cartesian diagram\begin{center}
$\xymatrix{
    \VV_Y(i^*\mcL) \ar[r]^{i'} \ar[d]^{p_Y}  & \VV_X(\mcL) \ar[d]^{p_X} \\
     Y \ar[r]^{i} & X
  }$ 
\end{center}
As $p_X$ is flat, by flat base change \cite[\href{https://stacks.math.columbia.edu/tag/02KH}{Tag 02KH}]{stacks-project}, we have an isomorphism between the coherent sheaves $i'_*p_Y^* \mcF$ and $p_X^*i_*\mcF$. Hence we have an isomorphism between $H^0(\VV_X(\overline{\mcL}),  p_X^*i_*\mcF )$ and $H^0(\VV_X(\overline{\mcL}),  i'_*p_Y^* \mcF)$.\\
Then, by \cite{Bost-Charles}, we have an isomorphism between $H^0(\VV_X(\mcL),T,  \iota'_*p_Y^* \mcF)$ and $H^0(\VV_Y(i^*\mcL),T\cap \VV_Y(i^*\mcL)(\CC),  p_Y^* \mcF)$. But, $T\cap \VV_Y(i^*\mcL)(\CC)$ is the tube associated to $i^*\overline{\mcL}$.\\
This concludes.\end{proof}
\subsubsection{Additivity with respect to exact sequences.} As a consequence of the additivity of the arithmetic degree, and the compatibility of the bornological structures with exact sequences of coherent sheaves in the semipositive case, we get the following result : 

\begin{prop}Let $X$ be a projective scheme over $\ZZ$ of dimension $d$. Let $\overline{\mcL}$ be a semipositive seminormed line bundle over $X$, ample over $\ZZ$. Let  $0 \to \mcE \to \mcF \to \mcG \to 0$ is an exact sequence of coherent sheaves over $X$. Let $r\geq d$. Then, 
    $$\overline{c}_r(X,\overline{\mcL},\mcF) = \overline{c}_r(X,\overline{\mcL},\mcE) + \overline{c}_r(X,\overline{\mcL},\mcG)$$ 
    $$\underline{c}_r(X,\overline{\mcL},\mcF) = \underline{c}_r(X,\overline{\mcL},\mcE) + \underline{c}_r(X,\overline{\mcL},\mcG)$$
    if $\underline{c}_r(X,\overline{\mcL},\mcE) = \overline{c}_r(X,\overline{\mcL},\mcE)$ or $\underline{c}_r(X,\overline{\mcL},\mcG) = \overline{c}_r(X,\overline{\mcL},\mcG)$.
\end{prop}

\begin{proof}
Let $p_X: \VV_X(\mcL)\to X$ be the total space of $\mcL^\vee$.\\
The short exact sequence of coherent sheaves 
$0 \to \mcE \to \mcF \to \mcG \to 0$
induces a strict exact sequence
$$0\to H^0(\VV_X(\overline{\mcL}),p_X^*\mcE) \to H^0(\VV_X(\overline{\mcL}),p_X^*\mcF) \to H^0(\VV_X(\overline{\mcL}),p_X^*\mcG)$$
where the last map has finite cokernel, see \cite{Bost-Charles}.\\
Then, this is an immediate corollary of definition \ref{invariantforgradedmodule} and lemma \ref{additivityclassical}. Indeed, because we are only concerned with the higher degrees, we can assume that the last map is an epimorphism. Then, if $(||\cdot||_j)_{j\in \NN}$ is a decreasing family of Hilbertian seminorms that defines the bornology on $H^0(\VV_X(\overline{\mcL}),p_X^*\mcF)$, then, the induced seminorms $(||\cdot||_{j,ind})_{j\in \NN}$ defines the bornology on $H^0(\VV_X(\overline{\mcL}),p_X^*\mcE)$ and the quotient seminorms $(||\cdot||_{j,q})_{j\in \NN}$ defines the bornology on $H^0(\VV_X(\overline{\mcL}),p_X^*\mcG)$. Then, we conclude by lemma \ref{additivityclassical} and by using classical inequalities for the limit sup of sums and the limit inf of sums.\\
\end{proof}

\subsubsection{Approximation by continuous metrics.}

Using the fact that upper semicontinuous seminorms can be approximated by above by continuous norms (see Yosida regularization in \cite{Giaquinta}, for example), we have an approximation lemma for arithmetic Hilbert invariants:

\begin{lemma}\label{approximationcontinuous}
Let $X$ be a projective scheme over $\ZZ$ of dimension $d$. Let $\overline{\mcL} = (\VV_X(\mcL), T)$ be a seminormed line bundle over $X$. Let $\mcF$ be a coherent sheaf over $X$. Let $r\geq d$.\\
Let $(T_j)_{j\in \NN}$ a decreasing family of closed neighborhood corresponding to continuous norms on $\mcL$. The associated Hermitian line bundles are denoted $\overline{\mcL_j}$.\\
Then, we have increasing limits: 
$$\overline{c}_r(X,\overline{\mcL}, \mcF) = \lim_j \overline{c}_r(X,\overline{\mcL_j}, \mcF) $$
$$\underline{c}_r(X,\overline{\mcL}, \mcF) = \lim_j \underline{c}_r(X,\overline{\mcL_j}, \mcF) $$
\end{lemma}

\begin{proof} This is a direct consequence of the definition of the arithmetic Hilbert invariants and the realisation of the bornology on $H^0(\VV_X(\overline{\mcL}), p_X^*\mcF)$ as given by the direct limit $\lim_j H^0(T_j,\mcF)$.\\
Firstly, $H^0(T_j,\mcF)\to H^0(T_{j+1},\mcF)$ is bounded.
Therefore, if $(||\cdot||_{j,i})_{i\in \NN}$ is a decreasing family of seminorms defining the bornology on $H^0(T_i,\mcF)$ then we can choose a family of seminorms $(||\cdot||_{j+1,i})_{i \in \NN}$ defining $H^0(T_{j+1},\mcF)$ such that $||\cdot||_{j,i} \geq ||\cdot||_{j+1,i} $ for all $n$.\\
This proves that the sequences $\overline{c}_r(X,\overline{\mcL_j}, \mcF)$ and $\underline{c}_r(X,\overline{\mcL_j}, \mcF)$ are increasing.\\
Furthermore, as $H^0(T,\mcF) = \lim_j H^0(T_j,\mcF)$, the decreasing family $(||\cdot||_{i,i})$ defines the bornology on $H^0(T,\mcF)$. As, $\widehat{\chi}(\cdot, ||\cdot||_{i,j})$ is increasing both in $i$ and $j$, this proves the statement.
\end{proof}

\section{Deformation to the projective completion of the cone}

Inspired by the key idea for the intersection theory developed in \cite{Baum1975RiemannrochFS} and \cite{fulton}, we will define an explicit deformation of projective schemes over $\ZZ$ endowed with a deformation of a semipositive seminormed line bundle.\\
In \cite{fulton}[Chapter 5] or in \cite{Baum1975RiemannrochFS} is introduced 
 the deformation to the normal cone, the construction described in this section is a variation on this construction.
The aim is to give a geometric interpretation of the classical technique of "dévissage".\\
We will first present the geometric construction and its properties and then this will allow us to describe the deformation of analytic tubes. \\
To a hyperplane section with respect to a very ample line bundle over a projective scheme, we will associate a projective scheme over $\PP^1$ and an ample line over this projective scheme. We will then describe the corresponding natural action of $\GG_m$, the description of the isotypic components in proposition \ref{decompSelonActionGm} is then to be understood as the link with the "dévissage" technique. \\
The relationship between this construction and the deformation to the normal cone will be described in the appendix.

\subsection{The geometric construction.}

\subsubsection{Definitions.}

Let $A$ be either a field $k$ or $\ZZ$. \\
Consider the following data :
\begin{itemize}
    \item  $s: X \to \Spec A$ a projective scheme over $A$.
    \item A closed immersion $i : Y \to X$.
    \item A line bundle $\mcL$ over $X$ very ample over $A$, such that the ideal $I_\bullet \subset \bigoplus H^0(X,\mcL^{\otimes n})$ associated to $i$ is generated in degree 1.
\end{itemize}

\paragraph{Terminology:} We say that $Y$ is a \emph{hyperplane section with respect to $\mcL$}.\\

Let $S_\bullet$ be the graded ring $\bigoplus\limits_{n\in \NN} H^0(X, \mcL^{\otimes n})$.\\
Then, let $D_YX$ be the projective scheme over $\PP^1_A = \Proj A[u,t]$, defined over $\AA^1_A = \{ u \neq \infty\} = \Spec A[u/t]$ by $$D_YX_{u\neq \infty} =X \times_A \AA^1_A $$ and over $D(u) = \{ u \neq 0\}$ by $$D_YX_{u\neq 0} = \Proj  S'_{\bullet }$$ where $S'_{\bullet }$ is the graded $A[t/u]$-algebra   $ \bigoplus\limits_{l\in \ZZ} I_{\bullet}^l(t/u)^{-l}$ (with $I_\bullet^{l}= S_\bullet$ for $l\leq 0$).\\[2mm]

\begin{rem}
Those two pieces glue well together as above $D(t,u)$ both of them are canonically isomorphic to $X\times_A \AA^1_A-\{0\}$, see \cite{matsumura_1987}[chapt. 15.4]. 
\end{rem}

\begin{rem} It is possible to define $D_YX$ for every closed immersion $i:Y \to X$ of projective schemes over $A$, however the fact that $Y$ is a hyperplane section with respect to $\mcL$ will be crucial to the definition of the ample line bundle over $D_YX$.\\
\end{rem}

\begin{defn}  The projective scheme $D_YX$ over $\PP^1_A$ is then called the \emph{total space of the deformation to the projective completion of the cone associated to $(X,Y,\mcL)$}.
\end{defn}

\paragraph{Notation:} If $Y= div(s)$ where $s$ is a global section of $\mcL$, then $Y$ is a hyperplane section with respect to $\mcL$ and we denote the projective completion of the cone associated to $(X,div(s),\mcL)$ by $D_sX$.\\

The following proposition justifies the terminology: 

\begin{prop}\label{fibrealinfini} Let $X$ be a projective scheme over $A$. Let $\mcL$ be a line bundle over $X$ very ample over $A$. Let $s$ be a regular global section of $\mcL$.\\
Then, \begin{itemize}
    \item The fiber above $u = 1$, $D_sX_{|1}$ is $X$.
    \item The fiber above $u = \infty$, $D_sX_{|\infty}$ is $ \PP(C_s\oplus 1)$ the projective completion of the cone $C_s$ over $div(s)$ associated to $\mcL_{|div(s)}$.\\
\end{itemize}
\end{prop}

\begin{rem} In this proposition, we use the definition of $\PP(C_X \oplus 1)$ given in \cite{fulton}[Appendix B] and \cite{grothendieck}[II-8]. Let $S_\bullet = \bigoplus_{n\in \NN} H^0(X,\mcL^{\otimes n})$ define the cone $C_X$ over $X$ then $ \PP(C_X \oplus 1) = \Proj  S_{\bullet}[y] $ where $(S_{\bullet}[y])_n = S_n \oplus S_{n-1}y \oplus \dots \oplus S_0y^n$. \end{rem}
\begin{proof}
Let $S_\bullet = \bigoplus_{n\in \NN} H^0(X,\mcL^{\otimes n})$, let $I_\bullet = sS_{\bullet}$.\\
The first statement is clear.\\
The fiber above infinity is given by $\Proj  S'_{\bullet }/(t/u) S'_\bullet$ where $S'_{\bullet } = \bigoplus_{l\in \ZZ} I_\bullet^l(t/u)^{-l}$.\\
Hence, $D_sX_{|\infty}= \Proj S''_\bullet$ where $S''_{\bullet } = \bigoplus_{l \geq 0} I_\bullet^l/I_\bullet^{l+1}$. And as $I_\bullet = s S_\bullet$ and that $s$ is regular, we have an isomorphism of graded algebra between $S''_\bullet$ and $(S_\bullet/I_\bullet) [y]$. \\
This concludes.\\\end{proof}

\begin{ex}\label{examplePN} Let $X = \PP^N_A = \Proj A[x_0,\dots,x_N]$, $\mcL = \mcO_\PP(1)$ and $Y = \mathrm{div}(x_0,\dots, x_{M-1})$. Then, $D_YX$ is isomorphic to the closed subscheme of $\PP^{N+M}_A\times \PP^1_A= \Proj A[x_0, \dots, x_N, y_0 \dots ,y_{M-1}] \times \Proj A[u,t]$ defined by the equations $ty_{i} - ux_{i} = 0, x_{j}y_{i} - y_{j}x_{i}=0$ for $0\leq i,j\leq M-1$.\\
The case where $Y$ is a hyperplane ($M=1$) shows that the fiber above $\infty$ is $\Proj A[x_1,\dots,x_N,y_0] = \PP^N_A$.\\
\end{ex}

The following proposition gives a convenient description of the deformation to the projective completion of the cone. It is, in particular, very linked with the previous example and the functoriality results which will be proven below.\\ 
\begin{prop}\label{immersionPn} Let $X$ be a projective scheme over $A$. Let $\mcL$ be a line bundle over $X$ very ample over $A$. Let $i:Y\to X$ be a closed immersion where the associated graded ideal $I_{\bullet}\subset \bigoplus H^0(X,\mcL^{\otimes n})$ is generated in degree $1$ by global sections $t_1,\dots,t_M$ of $\mcL$.\\
Assume that $s_0, \dots, s_N \in H^0(X,\mcL)$ induce a closed embedding into $\PP^N_A$.\\
Then, $D_YX$ is the schematic closure of the image of the map:$$\begin{array}{ccccc}
\varphi & :   X \times_A \AA^1_A   & \to & \PP^{N+M}_A \times_A \PP^1_A \\
  & (x , u) & \mapsto & ([s_0(x) : \dots : s_N(x):u t_1(x) : \dots ut_M(x)], [u:1]) 
  \end{array}$$
\end{prop}

\begin{proof}
Let $S_\bullet = \bigoplus_{n\in \NN} H^0(X,\mcL^{\otimes n})$.\\
Over $u\neq \infty$, the schematic closure is isomorphic to $X\times_A\AA^1_A$.\\
Over $u \neq 0$. Firstly, we have a surjection $A[x_0,\dots,x_N] \to S_\bullet$ which sends $x_i$ to $s_i$.\\
$\phi : X\times_A \AA^1_A-\{0\} \to \PP^{N+1}_A \times_A \PP^1_A-\{0\}$ corresponds to a morphism $A[u^{-1}]$-algebra from $A[x_0,\dots,x_N,y_1,\dots, y_M][u^{-1}]$ to  $S_{\bullet}[u,u^{-1}]$ sending $x_i$ to $s_i$ and $y_j$ to $ut_j$. This image of this morphism is $\bigoplus_{l\in \ZZ}I_\bullet^lu^l$.\\
Hence, the schematic closure is the closed immersion from $D_YX_{u\neq 0}$ to $\PP^{N+1}_A \times_A \PP^1_A-\{0\}$ corresponding to the morphism of $A[u^{-1}]$-graded algebra from $A[x_0,\dots,x_N,y_1, \dots, y_M][u^{-1}]$ to $\bigoplus_{l\in \ZZ}I_\bullet^lu^l$ which sends $x_i$ to $s_i$ and $y_j$ to $ut_j$. \\
This concludes.\\\end{proof}

\subsubsection{Deformation of the ample line bundle.}

Let $X$ be a projective scheme over $A$. Let $i:Y\to X$ a closed immersion. Let $\mcL$ be a line bundle over $X$ very ample over $A$ such that the graded ideal $I_{\bullet}\subset \bigoplus H^0(X,\mcL^{\otimes n})$ associated to $i$ is generated in degree $1$ by global sections $t_1,\dots,t_M$ of $\mcL$.
Assume that $s_0, \dots, s_N \in H^0(X,\mcL)$ induce a closed embedding into $\PP^N_A$.\\

Consider $D_YX$ the deformation to the projective completion of the cone associated with $(X,Y,\mcL)$. \\

By proposition \ref{immersionPn}, there is a closed immersion $\phi : D_YX \to \PP^{N+M}_A \times_A \PP^1_A$. Let $\mcL'$ be the pullback of $\mcO_{\PP^{N+M}}(1)\boxtimes \mcO_{\PP^1}(1)$ by $\phi$.\\
This defines a canonical line bundle on $D_YX$ ample over $A$.\\

\begin{prop}\label{independanceCanonical}
With the above notation, $\mcL'$ does not depend on the choice of $s_0, \dots , s_N \in H^0(X,\mcL)$ or on the choice of $ t_1, \dots, t_M \in H^0(X,\mcL)$.
\end{prop}

\begin{proof}
Let $s_0, \dots , s_N \in H^0(X,\mcL)$ and $s'_0, \dots , s'_{N'} \in H^0(X,\mcL)$ be two families of global sections inducing  closed immersions of $X$ into $\PP^N_A$ and $\PP^{N'}_A$ respectively. \\
Let $t_1, \dots , t_M \in H^0(X,\mcL)$ (resp. $t'_1, \dots , t'_{M'} \in H^0(X,\mcL)$) generating $I_{\bullet}$.\\
By considering the union of these two families, we can assume that $N'\geq N$, and $s_0 = s'_0, \dots, s_N = s'_N$ (resp. $M'\geq N$, and $t_1 = t'_1, \dots, t_M = t'_M$). Then, we
have a commuting diagram of closed immersions, given by the proposition \ref{immersionPn} : $$\xymatrix{
    D_YX \ar[rd]^{\phi'} \ar[r]^{\phi}  & \PP^{N+M}_A \times_A \PP^1_A \ar[d]^{\pi\times id} \\
     & \PP^{N'+M'}_A \times_A \PP^1_A
  }
$$
The closed immersion $\pi: \PP^{N+M}_A \to \PP^{N'+M'}_A$ is given by expressing $s'_i$ (resp $t'_i$) as a linear combination of $s_0, \dots , s_N$ (resp $t_1, \dots, t_M$). Then by functoriality of the pullbacks, this concludes.\\
\end{proof}

\begin{defn} With the above notation,
$(D_YX, \mcL')$ is called the \emph{deformation to the projective completion of the cone associated to $(X,Y,\mcL)$}. \\
\end{defn}

\begin{prop}\label{associatedamplelinebundle} Let $X$ be a projective scheme over $A$. Let $\mcL$ be a line bundle over $X$ very ample over $A$. Let $i:Y\to X$ be a hyperplane section with respect to $\mcL$.\\
Let $(D_YX, \mcL')$ be the deformation to the projective completion of the cone associated to $(X,Y,\mcL)$. Then,
\begin{itemize}
    \item the line bundle $\mcL'$ over $D_YX$ is ample over $A$,
    \item the restriction of $\mcL'$ over $1$ is the line bundle $\mcL$ on $X$,
    \item the restriction of $\mcL'$ over $\infty$ is the canonical line bundle $\mcO(1)$ associated to the generated in degree $1$ graded algebra $\bigoplus I_\bullet^l/I_\bullet^{l+1}$.
\end{itemize}
\end{prop}

\begin{proof}
As $\mcL'$ is the restriction of an ample line bundle to a closed subscheme, the first point is clear.\\

In the same setting as in proposition \ref{immersionPn},
there is a closed immersion $\phi : D_YX \to \PP^{N+M}_A \times_A \PP^1_A$. This closed immersion restricts over $1$ to a closed immersion $X\to \PP^{N+M}_A$ induced by global sections of $\mcL$. As the pullback of $\mcO_{\PP^{N+M}}(1)$ by this last closed immersion is $\mcL$, by definition of $\mcL'$,  this concludes for the second point.

The restriction of the closed immersion $\phi : D_YX \to \PP^{N+M}_A \times_A \PP^1_A$ to the fiber above infinity, is a closed immersion of projective schemes induced by the surjection $A[x_0,\dots,x_N,y_1,\dots, y_M] \to \bigoplus I_\bullet^l/I_\bullet^{l+1}$ sending $x_i$ to $s_i \in S_\bullet/I_\bullet$ and $y_j $ to $t_j \in I_\bullet/I^2_\bullet$. 
Hence the restriction of $\mcO_{\PP^{N+M}_A}(1)$ is indeed the canonical line bundle $\mcO(1)$ associated to the graded algebra  $\bigoplus I_\bullet^l/I_\bullet^{l+1}$ .
\end{proof}

\begin{rem}
If the graded algebra  $\bigoplus I_\bullet^l/I_\bullet^{l+1}$ is not generated in degree $1$, then the sheaf associated to the shifting in degree one of this algebra may not be a line bundle.
\end{rem}

\subsubsection{Functoriality}

\begin{prop}\label{functoriality}
Let $\pi : X\to V$ be a closed immersion of projective schemes over $A$. Let $\mcM$ be a line bundle over $V$ very ample over $A$. Let $j:W\to V$ be a hyperplane section with respect to $\mcM$.\\
Let $\mcL=\pi^*\mcM$ and $i:Y \to X$ be the pullback of $j$ by $\pi$. Then $i:Y \to X$ is a hyperplane section with respect to $\mcL$.\\
 Let $(D_WV, \mcM')$  (resp. $(D_YX, \mcL')$) be the deformation to the projective completion of the cone associated to $(V,W,\mcM)$ (resp. $(X,Y,\mcL)$).\\
Then we have a canonical closed immersion from $D_YX$ to $ D_WV$ above $\PP^1_A$ which verifies that :\begin{itemize}
    \item the pullback of $\mcM'$ by this closed immersion is $\mcL'$,
    \item this closed immersion restricts to the $\AA^1_A = {\PP^1_A}_{|u\neq \infty}$-morphism  $\pi \times id : X \times \AA^1_A \to V \times \AA^1_A$.
\end{itemize}
\end{prop}

\begin{proof}
Assume that the ideal $J_{\bullet} \subset \bigoplus_{n\in \NN} H^0(V,\mcM^{\otimes n})$ associated to the immersion $j:W\to V$ is generated in degree $1$ by section $t_1, \dots, t_{M} \in H^0(V,\mcM)$. \\
Then, the restrictions ${t_1}_{|X}, \dots, {t_{M}}_{|X} \in H^0(X,\mcL)$ generate the ideal $I_{\bullet} \subset \bigoplus_{n\in \NN} H^0(X,\mcL^{\otimes n})$ associated to the immersion $i:Y\to X$.\\  
Let $s_0,\dots , s_N \in H^0(V,\mcM)$ be sections inducing a closed immersion from $V$ to $\PP^N_A$.\\
Then, the restrictions ${s_0}_{|X}, \dots, {s_{N}}_{|X} \in H^0(X,\mcL)$ induce a closed immersion from $X$ to $\PP^N_A$.\\
Then, we have a commutative diagram : 
$$\xymatrix{
    X\times_A \AA^1_A \ar[r] \ar[d]^{\pi \times \text{id}}  & \PP^{N+M}_A\times_A \PP^{1}_A  \\
    V\times_A \AA^1_A \ar[ur]  & 
  }
$$
Taking the schematic closure of the two maps yields the closed immersion from $D_YX$ to $D_WV$ and the second point of the assertion.\\
Then, by functoriality of the pullback, as $\mcM'$ and $\mcL'$ are the pullback of the same line bundle on $\PP^{N+M}_A \times_A \PP^1_A$, the first point is clear. 
\end{proof}
\subsection{The action of $\GG_m$\\[2mm]}\label{introActionGm}

As the deformation to the normal cone, this deformation is endowed with an action of $\GG_m$. More precisely, we describe an action of $\GG_m$ on $(D_YX, \mcL')$, that is, an equivariant action of $\GG_m$ on $\VV_{D_YX}(\mcL') \to D_YX \to \PP^1_A$, where the action of $\GG_m$ on $\PP^1_A$ is the natural one. Furthermore, certain $\GG_m$-actions on $(X,\mcL)$ extend to $(D_YX, \mcL')$.

\subsubsection{Definitions.}
Let's assume that we are in the setting of proposition \ref{immersionPn}.\\

Consider the action of $\GG_m$ on $\PP^{N+M}_A\times \PP^1_A$ where $\lambda$ acts by sending \\$([x_0:\dots :x_{N+M} ],[ u:t])$ to $([x_0:\dots :x_N : \lambda x_{N+1}: \dots:\lambda x_{N+M}],[ \lambda u:t])$.\\
Then, this actions naturally extends to an action on $\VV_{\PP^{N+M}_A\times \PP^1_A}(\mcO_{\PP^{N+M}}(1)\boxtimes \mcO_{\PP^1}(1))$. It sends \\$([x_0:\dots : x_{N+M}],[u:t]),\mu (x_0,\dots , x_{N+M})\otimes(u,t))$ to \\
$([x_0:\dots :x_N : \lambda x_{N+1}: \dots:\lambda x_{N+M}],[ \lambda u:t]),\mu (x_0,\dots x_N, \lambda x_{N+1}, \dots ,\lambda x_{N+M})\otimes(\lambda u,t))$.\\

\begin{prop}\label{actionGm}
Let $X$ be a projective scheme over $A$, $\mcL$ be a line bundle over $X$ very ample over $A$, $i:Y\to X$ a hyperplane section with respect to $\mcL$.\\
Let $(D_YX, \mcL')$ be the deformation to the projective completion of the cone associated to $(X,Y,\mcL)$.\\
In the setting of proposition \ref{immersionPn}, there is a closed immersion $(D_YX, \mcL')$ to $(\PP^{N+M}_A \times_A \PP^1_A, \mcO_{\PP^{N+M}}(1)\boxtimes \mcO_{\PP^1}(1))$.\\
Then the action of $\mathbb{G}_m$ on $\PP^{N+M}_A\times \PP^1_A$ described above restricts to $D_YX$.\\
Furthermore, the action of $\GG_m$ on $\VV_{\PP^{N+M}\times \PP^1}(\mcO_{\PP^{N+M}}(1)\boxtimes \mcO_{\PP^1}(1))$ restricts to an action of $\GG_m$ on $\VV_{D_YX}(\mcL')$.
 \end{prop}

\begin{proof}
Consider the natural action of $\GG_m$ on $\AA^1_A$. Then the embedding $\phi : X\times \AA^1_A \to \PP^{N+M}_A \times \PP^1_A$ is equivariant by the action of $\GG_m$. Hence the schematic closure of its image is stable by the action of $\GG_m$.  \\
Then, by restriction, the second statement is clear.\\
\end{proof}

\begin{prop}
In the setting of proposition \ref{actionGm}, the action of $\GG_m$ on $D_YX$ does not depend on the choice of $t_1, \dots, t_M, s_0, \dots , s_N \in H^0(X,\mcL)$.\\
Furthermore, the action of $\GG_m$ on $\VV_{D_YX}(\mcL')$ does not depend on the choice of $t_1, \dots, t_M, s_0, \dots , s_N \in H^0(X,\mcL)$.
\end{prop}

\begin{proof}
Firstly, for two families of global sections of $H^0(X,\mcL^{\otimes k})$ defining two closed embedding $\phi: D_YX \to \PP^{N+M}_A \times_A \PP^1_A$ and $\phi' : D_YX \to \PP^{N'+M'}_A \times_A \PP^1_A$.\\
As in proposition \ref{independanceCanonical}, by considering the union of the two families, we can assume that there is a closed embedding from $\PP^{N+M}_A$ to $\PP^{N'+M'}_A$ making the following diagram commute :  
$$\xymatrix{
    D_YX  \ar[r]^{\phi} \ar[rd]^{\phi'}  & \PP^{N+M}_A\times_A \PP^{1}_A  \ar[d]\\
 & \PP^{N'+M'}_A\times_A \PP^{1}_A
  }
$$
Furthermore, this closed embedding is invariant by the action of $\GG_m$, so this concludes for the first assertion.\\
This diagram also induces a $\GG_m$-equivariant commutative diagram of closed embeddings: 
$$\xymatrix{
    \VV_{D_YX}(\mcL')  \ar[r] \ar[rd] & \VV_{\PP^{N+M}\times \PP^1}(\mcO_{\PP^{N+M}}(1)\boxtimes \mcO_{\PP^1}(1))  \ar[d]\\
 & \VV_{\PP^{N'+M'}\times \PP^1}(\mcO_{\PP^{N'+M'}}(1)\boxtimes \mcO_{\PP^1}(1))
  }
$$ proving the second assertion.\\
\end{proof}

\begin{defn}\label{definitionActionGm} With the above notation, this action of $\GG_m$ is called \emph{the natural transverse action of $\GG_m$ on $(D_YX,\mcL')$.}\\
\end{defn}

\subsubsection{Isotypic components.}
As the fiber at infinity is stable by the action of the natural action $\GG_m$ on $(D_YX, \mcL')$, we can relate an isotypic decomposition of the space of sections of the total space of the deformation to an isotypic decomposition of the space of sections on the fiber at infinity.

\begin{prop}\label{decompSelonActionGm} Let $X$ be a projective scheme over $A$. Let $\mcL$ be a line bundle over $X$ very ample over $A$. Let $i:Y\to X$ be a hyperplane section with respect to $\mcL$.\\
Let $(D_YX, \mcL')$ be the deformation to the projective completion of the cone associated to $(X,Y,\mcL)$ endowed with its natural transverse action of $\GG_m$.\\
Then, for $n\gg 0$,
\begin{itemize}
    \item $H^0(D_YX, \mcL'^{\otimes n})$ can be decomposed into subspaces $E_{l,n}$, where $\GG_m$ acts by the character $\lambda \mapsto \lambda^{n+l}$ for $ -n \leq l \leq n$.
    \item The fiber over $\infty$ is stable by the action of $\GG_m$, and hence, the action of $\GG_m$ restricts to $\VV_{D_YX_{\infty}}(\mcL')$. Furthermore, $H^0(D_YX_{\infty},\mcL'^{\otimes n})$ can be decomposed into subspaces $F_{l,n}$, where $\GG_m$ acts by the character $\lambda \mapsto \lambda^{n+l}$ for $ 0 \leq l \leq n$.
    \item  The restriction $r_1$ of $H^0(D_YX, \mcL'^{\otimes n})$ to the fiber over $1$ and the restriction $r_\infty$ of $H^0(D_YX, \mcL'^{\otimes n})$ to the fiber over $\infty$ induce commutative diagrams, for $0 \leq l \leq n$:
\begin{center}

$\xymatrix{
    E_{l,n} \ar[rd]^{r_\infty} \ar[r]^{r_1}  & r_1(E_{l,n}) \ar[d]^{\phi_{l,n}} \\
     & F_{l,n}
  }
$

\end{center}

\end{itemize}
Furthermore, for $n\gg0$, $r_1(E_{l+1,n}) \subset r_1(E_{l,n})$, $r_1(E_{0,n})= H^0(X, \mcL^{\otimes n})$ and
$\phi_{l,n}$ induce isomorphisms from $r_1(E_{l,n})/r_1(E_{l+1,n})$ to $F_{l,n}$.
\end{prop}

\begin{proof} Let $S_{\bullet} = \bigoplus\limits_{n\in \NN} H^0(X,\mcL^{\otimes n})$. 

Firstly, we have, for $n\gg 0$, 
\begin{center}
    $H^0(D_YX_{|u \neq \infty }, \mcL'^{\otimes n}) = t^{n}S_{n}[\frac{u}{t}]$\\
    $H^0(D_YX_{|u \neq \infty, u\neq 0 }, \mcL'^{\otimes n}) = t^{n}S_{n}[\frac{u}{t},  \frac{t}{u}]$\\
    $H^0(D_YX_{|u\neq 0 }, \mcL'^{\otimes n}) = u^{n}\bigoplus_{l\in \ZZ}(I_{\bullet})^l_{n}(\frac{u}{t})^{l}$
\end{center}
where the restriction maps are the natural ones.\\
Thus, for $n \gg 0$, we have $H^0(D_YX, \mcL'^{\otimes n}) =u^{n}\bigoplus_{-n\leq l \leq n}(I_{\bullet})^l_{n}(\frac{u}{t})^{l} $.\\

Furthermore, by using the immersion into $\PP^{N+M}_A\times_A \PP^{1}_A$ described in proposition \ref{immersionPn} induced, we have a surjection, for $n\gg 0$:
$H^0(\PP^{N+M}_A\times_A \PP^{1}_A, \mcO_{\PP^{N'+M'}}(n)\boxtimes \mcO_{\PP^1}(n)) \to H^0(D_YX, \mcL'^{\otimes n})$
which sends $P_1(x_0,\dots, x_N,y_1, \dots, y_M)\otimes P_2(u,t)$ to $P_1(s_0, \dots , s_N, t_1\frac{u}{t}, \dots, t_M\frac{u}{t})P_2(u,t)$, where $P_1,P_2$ are homogeneous polynomial of degree $n$.\\
Then as, $\GG_m$ acts on $H^0(\PP^{N+M}_A\times_A \PP^{1}_A, \mcO_{\PP^{N'+M'}}(n)\boxtimes \mcO_{\PP^1}(n))$ by sending $P_1(x_0,\dots, x_N,y_1, \dots, y_M)\otimes P_2(u,t)$ to $P_1(x_0,\dots, x_N,\lambda y_1, \dots, \lambda y_M)\otimes P_2(\lambda u,t)$.\\
Therefore, $\GG_m$ acts on $u^{n}(I_{\bullet})^l_{n}(\frac{u}{t})^{l} $ by the character $\lambda \mapsto \lambda^{n+l}$.\\
This concludes for the first point.\\

Secondly, $\mcL'$ is ample over $A$. Therefore, for $n\gg 0$, the restriction to the fiber over infinity $r_{\infty} : H^0(D_YX, \mcL'^{\otimes n}) \to H^0(D_YX_{\infty}, \mcL'^{\otimes n})$ is onto.\\
For $n \gg 0$, by quotienting by $t/u = 0$, this restriction is the surjection $u^{n}\bigoplus_{-n \leq l \leq n} (I_{\bullet})^l_n (t/u)^{-l} \to \bigoplus_{0 \leq l \leq n} (I_{\bullet})^l_n/(I_{\bullet})^{l+1}_n$.\\
And as $\infty$ is stable by the action of $\GG_m$ on $\PP^1_A$. The fiber over $\infty$ is stable by the action of $\GG_m$. Hence, $\VV_{D_YX_{\infty}}(\mcL')$ is endowed with an action of $\GG_m$ for which $r_\infty$ is equivariant. \\
This proves the second point.\\

Finally, the restriction to the fiber over $1$, $r_1: H^0(D_YX, \mcL'^{\otimes n}) \to H^0(X,\mcL^{\otimes n})$ sends $u^{n}(I_{\bullet})^l_{n}(\frac{u}{t})^{l}$ to $(I_{\bullet})^l_{n} \subset H^0(X,\mcL^{\otimes n})$. This concludes for the rest of the proposition.\\
\end{proof}

As a consequence of this demonstration, on a less complete description but for all $n\in \NN$, sections in $H^0(X,\mcL^{\otimes n})$ can be extended to sections in $H^0(D_YX,\mcL'^{\otimes n})$ by using the natural transverse action of $\GG_m$.

\begin{prop}\label{extensionDesFonctions} Let $X$ be a projective scheme over $A$. Let $\mcL$ be a line bundle over $X$ very ample over $A$. Let $i:Y\to X$ be a hyperplane section with respect to $\mcL$.\\
Let $(D_YX, \mcL')$ be the deformation to the projective completion of the cone associated to $(X,Y,\mcL)$ endowed with its natural transverse action of $\GG_m$.\\
Then for all $n\in \NN$, $s\in H^0(X,\mcL^{\otimes n})$, there is $s'\in H^0(D_YX, \mcL'^{\otimes n})$ which restricts to $s$ and on which $\GG_m$ acts by the character $\lambda \mapsto \lambda^n$.
\end{prop}

\begin{proof}As in the previous proof, let $S_{\bullet} = \bigoplus\limits_{n\in \NN} H^0(X,\mcL^{\otimes n})$. \\

Taking into account that $\mcL$ is very ample, we have, for all $n$, 
\begin{center}
    $H^0(D_YX_{|u \neq \infty }, \mcL'^{\otimes n}) = t^{n}S_{n}[\frac{u}{t}]$\\
    $H^0(D_YX_{|u \neq \infty, u\neq 0 }, \mcL'^{\otimes n}) = t^{n}S_{n}[\frac{u}{t},  \frac{t}{u}]$\\
    $H^0(D_YX_{|u\neq 0 }, \mcL'^{\otimes n}) \supset u^{n}S_n$
\end{center}
Hence, as in the previous proof, this leads to global sections of $\mcL'^{\otimes n}$ on which $\GG_m$ acts by the character $\lambda \mapsto \lambda^n$, for which the restriction on the fiber above is $1$ is the space of sections $H^0(X,\mcL^n)= S_n$.\\
\end{proof}

\subsubsection{Compatibility with other $\GG_m$-actions.} Under some conditions, $\GG_m$-actions on $X$ can be extended trivially to $\GG_m$-actions on $D_YX$. 

\begin{prop} \label{extensionaction}
Let $X$ be a projective scheme over $A$. Let $\mcL$ be a line bundle over $X$ very ample over $A$. Let $i:Y\to X$ be a closed immersion. Assume that the ideal $I_\bullet \subset \bigoplus H^0(X,\mcL^{\otimes n})$ associated to the closed immersion $i$ is generated in degree one by global sections $t_1,\dots, t_M \in H^0(X,\mcL)$.\\
Let $\psi : X \to \PP^N_A$ be a closed immersion induced by global sections of $\mcL$. \\
Assume that there is an action of $\GG_m$ on $\VV_{\PP^N_A}(\mcO_{\PP^N}(1))$ which restricts to $\VV_{X}(\mcL)$.\\
Assume that $t_1,\dots, t_M$ are invariant by this action.\\
Let $(D_YX, \mcL')$ be the deformation to the projective completion of the cone associated to $(X,Y,\mcL)$.\\
Then, this action can be extended to an action on $\VV_{D_YX}(\mcL')$. This action acts on the fibers of $\VV_{D_YX}(\mcL')\to \PP^1_A$.
\end{prop}

\begin{proof}
Let $\phi : X\times \AA^1_A \to \PP^{N+M}_A\times \PP^1_A$ the map described in proposition \ref{immersionPn}. \\
The action of $\GG_m$ can be extended to $\VV_{\PP^{N+M}_\ZZ}(\mcO_\PP(1))$ by acting trivially on the last coordinates.\\
Through the trivial extension of this action of $\GG_m$ to $X\times \AA^1_\ZZ$ and $\PP^{N+M}_\ZZ\times \PP^1_\ZZ$, the map $\phi $ is equivariant. \\
Hence, the schematic closure $D_YX$ is stable by the action of $\GG_m$, and the action of $\GG_m$ on $\VV_{\PP^{N+M}_\ZZ\times \PP^1_\ZZ}(\mcO_{\PP^{N+M}_\ZZ}(1) \boxtimes \mcO_{\PP^{1}_\ZZ}(1))$ restricts to $\VV_{D_YX}(\mcL')$.

\end{proof}

Let consider deformations to the projective completion of the cone associated to a global section of $\mcL$. Thanks to the functioriality result \ref{functoriality} and to the deformation of $\PP^N_A$ with respect to a hyperplane \ref{examplePN} into another copy $\PP^N_A$, we can keep track of closed closed immersions to $\PP^N_A$. We deform  $(\pi:X\to \PP^N_A, div(s), \mcL=\pi^*\mcO_\PP(1))$ into $(\pi':D_sX_{\infty}\to \PP^N_A, \pi'^*\mcO_\PP(1))$. One the key feature is that the natural transverse action of $\GG_m$ on $(D_sX, \mcL')$ adds an action of $\GG_m$ and that $D_sX_\infty$ inherits the existing ones on $(X,\mcL)$:\\

\begin{cor}\label{CombinaisonActionGm}
Let $X$ be a projective scheme over $A$, $\mcL$ be a line bundle over $X$ very ample over $A$.\\
Let $\pi :X \to \PP^N_A = \Proj A[x_0, \dots , x_N]$ be a closed immersion induced by global sections of $\mcL$.
Let $s = \pi^*x_N$.\\
Let $(D_sX, \mcL')$ be the deformation to the projective completion of the cone associated to $(X,div(s),\mcL)$.\\
Assume that $\GG_m$ acts on $\VV_{\PP^N_\ZZ}(\mcO_\PP(1))$ by multiplication on the first variable $x_0$ of $\PP^{N}_\ZZ$ and that this action restricts to $\VV_{X}(\mcL)$. \\
Then, this action naturally extends to $\VV_{\PP^{N+1}_\ZZ \times \PP^1_\ZZ}(\mcO_{\PP^{N+1}}(1)\boxtimes \mcO_{\PP^1}(1))$ and restricts to $\VV_{D_sX}(\mcL')$.\\
At the fiber over infinity, there is a closed immersion $\pi' : D_sX_\infty \to \PP^N_\ZZ = \Proj A[x_0, \dots , x_{N-1}, y]$ equivariant for the action of $\GG_m$ by the multiplication on the first variable.\\ 
Moreover, the natural transverse action of $\GG_m$ on $D_sX$ restricts at the fiber at infinity to the restriction through $\pi'$ of the multiplication on last variable of $\PP^N_\ZZ$.

\end{cor}

\begin{proof}
As described in proposition \ref{immersionPn}, the closed immersion $X \to \PP^{N}_A = \Proj A[x_0,\dots,x_N]$ induces a closed immersion  $D_sX \to \PP^{N+1}_A \times_A \PP^{1}_A = \Proj A[x_0, \dots , x_N, y] \times_A \PP^{1}_A$ which induces an closed immersion $D_sX_{\infty} \to \PP^{N}_A = \Proj A[x_0, \dots , x_{N-1}, y]$ by functoriality as described in proposition the \ref{functoriality} and the example \ref{examplePN} . \\
Via the natural transverse action of $\GG_m$ on the total space of the deformation, $D_sX_{\infty}$ is stable by the action of $\GG_m$ by multiplication on the last variable. \\
Now, $s=\pi^*x_N$ is stable by the restriction of the action of $\GG_m$ by multiplication on the first variable, as $x_N$ is stable by the action of multiplication on $x_0$ and as the action restricts to $X$.\\
By the previous proposition, $D_sX$ is also stable by the action of $\GG_m $ on the first variable, which proves that $D_sX_{\infty}$ is also stable by the action of $\GG_m $ on the first variable.\\
\end{proof}

\subsubsection{Functoriality.}

The functoriality given in proposition \ref{functoriality} is compatible with the natural transverse $\GG_m$-action.

\begin{prop}\label{functorialityGm}
Let $\pi : X\to V$ be a closed immersion of projective schemes over $A$. Let $\mcM$ be a line bundle over $V$ very ample over $A$. Let $j:W\to V$ be a hyperplane section with respect to $\mcM$.
Let $\mcL=\pi^*\mcM$ and $i:Y \to X$ be the pullback of $j$ by $\pi$.\\
 Let $(D_WV, \mcM')$  (resp. $(D_YX, \mcL')$) be the deformation to the projective completion of the cone associated to $(V,W,\mcM)$ (resp. $(X,Y,\mcL)$).\\
By proposition \ref{functoriality}, we have a closed immersion from $D_YX$ to $ D_WV$.
This closed immersion is $\GG_m$-equivariant, restricts to the $\GG_m$-equivariant morphism $\pi \times id : X \times \AA^1_A \to V \times \AA^1_A$ and induces a $\GG_m$-equivariant closed immersion from $\VV_{D_XY}(\mcL')$ to $\VV_{D_WV}(\mcM')$.
\end{prop}

\begin{proof}
The following commutative diagram of the proof of proposition \ref{functoriality}:  
$$\xymatrix{
    X\times_A \AA^1_A \ar[r] \ar[d]^{\pi \times \text{id}}  & \PP^{N+M}_A\times_A \PP^{1}_A  \\
    V\times_A \AA^1_A \ar[ur]  & 
  }
$$ 
is $\GG_m$-equivariant.\\
Again, taking the schematic closures concludes for the first points.\\
Restricting the action of $\GG_m$ on $\VV_{\PP^{N+M}_A\times_A \PP^1_A}(\mcO(1)\boxtimes \mcO(1))$ to $\VV_{D_XY}(\mcL')$ and $\VV_{D_WV}(\mcM')$ concludes for the last point. \\
\end{proof}

\section{Deformation of the analytic tubes}

If $A=\ZZ$, and if $\mcL$ is endowed with a structure of seminormed line bundle, then $\mcL$ is ample over $\ZZ$, this allows to use the notion of equilibrium seminorm to describe a semipositive seminormed structure on the line bundle $\mcL'$ over $D_YX$.

\subsection{Definitions.}

Let $X$ be a projective scheme over $\ZZ$. Let $\overline{\mcL} = (\VV_X(\mcL), T)$ be a seminormed line bundle on $X$ such that the underlying line bundle is very ample over $\ZZ$. Let $i:Y\to X$ be a hyperplane section with respect to $\mcL$.\\
Let $(D_YX, \mcL')$ be the deformation to the projective completion of the cone associated to $(X, Y,\mcL)$.\\
Consider the seminorm $K$ which is the union of the zero section in $\VV_{D_YX}(\mcL')(\CC)$ and the orbit of $T\times \{[1:1]\} \subset \VV_{D_YX}(\mcL')_{|1}(\CC) $ under the restriction of the natural transverse action to $U(1)$ the unit complex numbers.\\
As $\mcL'$ is ample, we can take the equilibrium seminorm associated with $K$, that is the holomorphically convex envelop $T' = \widehat{K}\subset \VV_{D_YX}(\mcL')(\CC)$. This endows $\mcL'$ with a structure of semipositive seminormed line bundle $\overline{\mcL'} = (\VV_{D_YX}(\mcL'),T')$ on $D_YX$.\\
Notice that $T' = \widehat{K} = \widehat{U(1)\cdot T}$

\begin{defn}\label{tubeDefinition}
With the above notation, we say that $(D_YX, \overline{\mcL'})$ is the \emph{deformation to the projective completion of the cone associated to $(X,Y, \overline{\mcL})$}.\\
\end{defn}

The following proposition proves that $(D_YX,\overline{\mcL'})$ is a deformation of $(X, \overline{\mcL})$.\\

\begin{prop}
Let $X$ be a projective scheme over $\ZZ$. Let $\overline{\mcL}$ be a semipositive seminormed line bundle over $X$, very ample over $\ZZ$. Let $i:Y\to X$ be a hyperplane section with respect to $\mcL$.\\
Let $(D_YX,\overline{\mcL'})$ the deformation to the projective completion of the cone associated to $(X,Y,\overline{\mcL})$.\\
Then the fiber over $1$ of $(D_YX, \overline{\mcL'})$ is $(X,\overline{\mcL})$.
\end{prop}

\begin{proof} Let $\overline{\mcL'} = (\VV_{D_YX}(\mcL'),T')$ and  $\overline{\mcL} = (\VV_{X}(\mcL),T)$.\\
Firstly, by definition, $T' \supset T\times [1:1]$.\\
By proposition \ref{decompSelonActionGm}, functions $f\in H^0(\VV_X(\mcL), \mcO_\VV)$ extend to functions $f'\in H^0(\VV_{D_sX}(\mcL'), \mcO_\VV)$ on which $\GG_m$ acts by the character $\lambda \mapsto \lambda^{n}$ on the homogeneous functions $H^0(X,\mcL^{\otimes n})\subset H^0(\VV_X(\mcL), \mcO_\VV)$. More precisely, $f\in H^0(\VV_X(\mcL), \mcO_\VV)$ can be decomposed into a linear combination $\sum s_n$ where $s_n\in H^0(X,\mcL^{\otimes n}) $ and the extension $f' = \sum s'_n $ with $s'_n\in H^0(D_YX,\mcL'^{\otimes n}) $ verifies that the natural transverse action of $\GG_m$ acts on it by $\lambda \cdot f' = \sum \lambda^n s'_n$.\\
Now as $T$ is stable by the action of $U(1)$ on the fibers of $\mcL$ and that this action also acts on $f = \sum s_n$ by $\lambda \cdot f = \sum \lambda^n s_n$.
Thus, if $f\in H^0(\VV_X(\mcL), \mcO_\VV)$ extends to $f'\in H^0(\VV_{D_YX}(\mcL'), \mcO_\VV)$, $\sup\limits_{T}|f| = \sup\limits_{U(1)\cdot T}|f'| = \sup\limits_{T'} |f'|$.\\
As, $T$ is holomorphically convex, this proves that if $x\notin T\times [1:1]$, then $x\notin T'$. This concludes.\\
\end{proof}

\subsection{Functoriality.}

The functoriality property with respect to closed subscheme described in proposition \ref{functoriality} is compatible with the archimedean data defined above.\\

\begin{prop}\label{functorialityTube}
Let $\pi : X\to V$ be a closed immersion of projective schemes over $A$. Let $\overline{\mcM}$ be a seminormed line bundle over $V$ such that the underlying line bundle is very ample over $\ZZ$. Let $j:W\to V$ be a hyperplane section with respect to $\mcM$.\\
Let $\overline{\mcL} = \pi^*\overline{\mcM}$ and $i:Y \to X$ be the pullback of $j$ by $\pi$.\\
Let $(D_WV, \overline{\mcM'})$ and $(D_YX, \overline{\mcL'})$ be, respectively, the deformations to the projective completion of the cone associated with $(V,W, \overline{\mcM})$ and $(X,Y, \overline{\mcL})$. \\
Let $\pi' : D_YX \to D_WV$ the closed immersion described in \ref{functoriality}. \\ 
Then we have $\overline{\mcL'} = \pi'^*\overline{\mcM'} $.
\end{prop}

\begin{proof} Let $\overline{\mcL} = (\VV_X(\mcL),T)$ and $\overline{\mcM} = (\VV_V(\mcM),K)$.\\
By proposition \ref{functorialityGm}, we have a closed immersion $\pi' : \VV_{D_WV}(\mcM') \to \VV_{D_YX}(\mcL')$ which is $\GG_m$-equivariant.
As, over $u\neq \infty$, the $\GG_m$-equivariant $\pi'$ restricts to $\pi \times id : X \times \AA^1_A \to Y \times \AA^1_A$, we have $\pi'^{-1}(U(1)\cdot K) = U(1)\cdot T$. \\
$\pi'$ is a closed immersion, therefore $$\pi'^{-1}(\displaystyle\widehat{U(1)\cdot K  }) = \widehat{\pi'^{-1}(U(1)\cdot K )} = \widehat{U(1)\cdot T }$$
\end{proof}

\subsection{Action of $U(1)$ on the tube.}

This definition of the analytic tube on $\mcL'$ is compatible with the natural transverse action of $\GG_m$, in the sense that the analytic tube is stable by the restriction of the action to $U(1)$, the unit complex numbers.\\

\begin{prop}\label{stabilityTube}
Let $X$ be a projective scheme over $\ZZ$. Let $\overline{\mcL} = (\VV_X(\mcL), T)$ be a seminormed line bundle on $X$ such that the underlying line bundle is very ample over $\ZZ$. Let $i:Y\to X$ be a hyperplane section with respect to $\mcL$.\\
Let $(D_YX, \overline{\mcL'})$ be the deformation to the projective completion of the cone associated to $(X,Y, \overline{\mcL})$.\\
Consider the natural transverse action of $\GG_m$ on $\VV_{D_YX}(\mcL')$.
Then, the analytic tube $T'$ associated to $\overline{\mcL'}$ is stable by the restriction of this action to $U(1)$.
\end{prop}

\begin{proof}
Let $K$ be the orbit of $T\times \{[1:1]\} \subset \VV_{D_YX}(\mcL')_{|1}(\CC) $ under the restriction of the natural action to $U(1)$ the unit complex numbers.\\
The orbit of $T\times \{[1:1]\}$ under $U(1)$ is stable by the action of $U(1)$, by definition.\\
Then as $\lambda \cdot \widehat{K} = \widehat{\lambda\cdot K}$, this proves that $T' = \widehat{K}$ is stable by $U(1)$.\\
\end{proof}

\begin{rem}
When we refer to the previous action of $\GG_m$ and $U(1)$, we will systematically refer to them as the natural transverse action of $\GG_m$ or the natural transverse action of $U(1)$.\\
\end{rem}

\begin{prop}\label{extensionactionU1}
Let $X$ be a projective scheme over $\ZZ$. Let $\overline{\mcL} = (\VV_X(\mcL), T)$ be a seminormed line bundle on $X$ such that the underlying line bundle is very ample over $\ZZ$. Let $i:Y\to X$ be a closed immersion. Assume that the ideal $I_\bullet \subset \bigoplus H^0(X,\mcL^{\otimes n})$ is generated in degree one by global sections $t_1,\dots, t_M \in H^0(X,\mcL)$.\\
Let $\psi : X \to \PP^N_A$ be a closed immersion induced by global sections of $\mcL$. \\
Assume that there is an action of $\GG_m$ on $\VV_{\PP^N_\ZZ}(\mcO(1))$ which restricts to $\VV_{X}(\mcL)$ and such that $T$ is stable by the action of $U(1)$.\\
Assume that $t_1,\dots, t_M$ are invariant by this action.\\
Let $(D_YX, \overline{\mcL'})$ be the deformation to the projective completion of the cone associated to $(X,Y, \overline{\mcL})$.\\
By proposition \ref{extensionaction}, this action can be extended to an action on $\VV_{D_YX}(\mcL')$ which acts on the fibers of $\VV_{D_YX}(\mcL')\to \PP^1_\ZZ$.\\
Furthermore, the analytic tube $T'$ associated to $\overline{\mcL'}$ is stable by the restriction of this action of $U(1)$.
\end{prop}

\begin{proof}
By the proof of proposition \ref{extensionaction}, through the identification of $D_YX_{u\neq \infty}$ and $X\times \AA^1_\ZZ$, this action of $\GG_m$ extends trivially the action of $\GG_m$ on $X$. Thus, through the  identification of $\VV_{D_YX_{u\neq \infty}}(\mcL')$ and $\VV_{X\times \AA^1_Z}(\mcL)$, the orbit of $T\times \{[1:1]\}$ under the natural transverse action of $U(1)$ is $T \times U(1)$ and is therefore stable by the action of $U(1)$ because $T$ is stable by $U(1)$. \\
Therefore, if $K$ the orbit of $T\times \{[1:1]\} \subset \VV_{D_YX}(\mcL')_{|1}(\CC) $ under the restriction of the natural transverse action to $U(1)$, then $K$ is stable by $U(1)$.\\
Hence, $T' = \widehat{K}$ is stable by this action of $U(1)$.
\end{proof}

Considering deformations to the projective completion of the cone associated to a global section of $\mcL$ and keeping track of a closed immersion to $\PP^N_\ZZ$, we deform $(\pi:X\to \PP^N_A, \overline{\mcL})$ into $(\pi':D_sX_{\infty}\to \PP^N_A, \overline{\pi'^*\mcO_\PP(1)})$ and one key feature is that the natural transverse action of $\GG_m$ on $(D_YX, \mcL')$ and the natural transverse action of $U(1)$ on $T'$ adds an action of $\GG_m$ and $U(1)$ to the existing ones:\\

\begin{cor}\label{CombinaisonActionGmU1}
Let $X$ be a projective scheme over $\ZZ$. Let $\overline{\mcL} = (\VV_X(\mcL), T)$ be a seminormed line bundle on $X$ such that the underlying line bundle is very ample over $\ZZ$.\\
Let $\pi :X \to \PP^N_\ZZ = \Proj \ZZ[x_0, \dots , x_N]$ be a closed immersion induced by global sections of $\mcL$.
Let $s = \pi^*x_N$.\\
Let $(D_sX, \overline{\mcL'}) = (D_sX, (\VV_{D_sX}(\mcL'), T'))$ be the deformation to the projective completion of the cone associated to $(X, div(s),\overline{\mcL})$.\\
Assume that $\GG_m$ acts on $\VV_{\PP^N_\ZZ}(\mcO_\PP(1))$ by multiplication on the first variable $x_0$ of $\PP^{N}_\ZZ$, that this action restricts to $\VV_{X}(\mcL)$ and that $T$ is stable by $U(1)$. \\
Then, this action naturally extends to $\VV_{\PP^{N+1}_\ZZ \times \PP^1_\ZZ}(\mcO_{\PP^{N+1}}(1)\boxtimes \mcO_{\PP^1}(1))$, restricts to $\VV_{D_sX}(\mcL')$ and $T'$ is stable by the restriction of this action to $U(1)$.\\
At the fiber over infinity, there is a closed immersion $\pi' : D_sX_\infty \to \PP^N_\ZZ = \Proj \ZZ[x_0,\dots,x_{N-1},y]$ equivariant for the action of $\GG_m$ by the multiplication on the first variable. \\
Moreover, the natural transverse action of $\GG_m$ on $D_sX$ at the fiber at infinity is the restriction through $\pi'$ of the multiplication on last variable of $\PP^N_\ZZ$.\\
Furthermore, if $\overline{\mcO_\PP(1)} = (\VV_{D_sX_\infty}(\mcO_\PP(1)), T'')$ is the restriction of $\overline{\mcL'}$ to the fiber at infinity, then $T''$ is stable by the restriction of these two actions to $U(1)$.
\end{cor}

\begin{proof}
This is a direct consequence of corollary \ref{CombinaisonActionGm} and proposition \ref{extensionactionU1}.
\end{proof}

\subsection{Conservation of the uniformly definite property.}

We show in this paragraph that the deformation to the projective completion of the cone deforms a uniformly definite normed line bundle into a uniformly definite normed line bundle.\\

\begin{prop}\label{stongdeftostrongdef}
Let $X$ be a projective scheme over $\ZZ$. Let $\overline{\mcL}$ be a uniformly definite semipositive seminormed line over $X$, very ample over $\ZZ$. Let $i:Y\to X$ be a closed immersion. Let $i:Y\to X$ be a hyperplane section with respect to $\mcL$. \\
Let $(D_YX,\overline{\mcL'})$ the deformation to the projective completion of the cone of $(X,Y,\overline{\mcL})$.\\
Assume furthermore that $\overline{\mcL}$ is uniformly definite. Then the seminormed line bundle $\overline{\mcL'}$ over $D_YX$ is uniformly definite.
\end{prop}

\begin{proof}
According to proposition \ref{functorialityTube} and by the definition of the uniformly definite property, we can assume that $X = \PP^N_\ZZ = \Proj \ZZ[x_0, \dots, x_N]$, $\mcL = \mcO(1)$, $Y = div(x_0,\dots, x_{M-1})$ and that the metric on $\mcO(1)$ is definite continuous. As this metric is then bigger than a multiple of the Fubini-Study metric, and the deformation of a multiple of a metric is the multiple of the deformation of the metric, we can assume that the metric is the Fubini-Study metric (taking the holomorphically convex hull commutes with the multiplication along the fibers). \\
Then, we have to show that the tube of deformation of the Fubini-Study metric contains a neighborhood of the zero section of $\VV_{D_YX}(\mcL)(\CC)$. Let $K \subset \VV_X(\mcL)(\CC) = \VV_{D_YX}(\mcL')(\CC)_{|u=1} $ be the tube corresponding to the Fubini-Study metric, the tube associated with the deformation is $\widehat{U(1)\cdot K}$ where we use the natural transverse action of $\GG_m$ described in paragraph \ref{introActionGm}.\\
For convenience in the following computation, we choose to scale the Fubini-Study by $\sqrt{N+1}$.\\ 
Then, $D_YX$ is the closed subscheme of $\PP^{N+M}_\ZZ\times \PP^1_\ZZ= \Proj \ZZ[x_0, \dots, x_N, y_0 \dots ,y_{M-1}] \times \Proj \ZZ[u,t]$ defined by the equations $ty_{i} - ux_{i} = 0, x_{j}y_{i} - y_{j}x_{i}=0$ for $0\leq i,j\leq M-1$.
We have $$H^0(\VV_{D_YX}(\mcL'), \mcO_\VV) = \frac{\ZZ[x_0 t, \dots, y_{M-1} t, x_0 u, \dots , y_{M-1} u]}{(y_{i}t - x_{i}u, x_{j}y_{i} - y_{j}x_{i})}$$
Let $P = \sum_j a_{\underline{n_j},k_j}\underline{x}^{\underline{n_{j,\bullet}}}\underline{y}^{\underline{n_{j,N+\bullet}}}u^{k_j}t^{|\underline{n_j}|-k_j} \in H^0(\VV_{D_YX}(\mcL'), \mcO_\VV)$ where the last coordinates $n_{j,N+1},\dots, n_{j,N+M}$ are equal to zero if $k_j \neq |\underline{n_j}|$, where $ |\underline{n_j}|$ denote the sum of the coordinates, and where, if $i<i'$, $n_{j,N+i'}=0$ or $n_{j,i-1}=0$.\\

Then, by integrating over $\{(x_0 = e^{i\theta_0}, \dots , x_N = e^{i\theta_N}, u = e^{i\theta'}, t =1 , y_{i} = e^{i(\theta_{i}+\theta')}  )\}\subset U(1)\cdot K$ we get that $$\sup_{z \in U(1)\cdot K } |P(z)| \geq \left(\sum_j |a_{\underline{n_j},k_j}|^2\right)^{1/2}$$
Now, by comparing the $L^2$-norms and $L^1$-norms in the finite case, we can prove that there is a $\epsilon >0$ such that $\widehat{U(1)\cdot K} \supset K_\epsilon$ where $K_\epsilon= \{|ux_i|<\epsilon, |tx_i|< \epsilon, 0\leq i \leq N \} $.\\
To be more precise, we can prove that $$\sum\limits_{i\geq 0, |\underline{n}|=i} |b_{\underline{n}}|^2 \geq \frac{1}{2}\left(\sum\limits_{i\geq 0, |\underline{n}|=i}|b_{\underline{n}}|\epsilon^i \right)^2$$
for $\epsilon\leq 1/\sqrt{2(N+1)}$, for all $(b_{\underline{n}})_{\underline{n}\in \NN^{N+1}}$ with finite support.
Indeed, 
$$\sum\limits_{i\geq 0, |\underline{n}|=i} |b_{\underline{n}}|^2 - \frac{1}{2}\left(\sum\limits_{i\geq 0, |\underline{n}|=i}|b_{\underline{n}}|\epsilon^i \right)^2 = \sum\limits_{i \geq 0}\left(2^{i-1}\left(\sum_{|\underline{n}|=i}|b_{\underline{n}}|\epsilon^{|\underline{n}|}-\sum_{|\underline{n}|>i}|b_{\underline{n}}|\epsilon^{|\underline{n}|}\right)^2 + \left(\sum_{|\underline{n}|=i}|b_{\underline{n}}|^2 - 2^i\epsilon^{2i}(\sum_{|\underline{n}|=i}|b_{\underline{n}}|)^2 \right)\right)$$
Hence, this proves that there is $\epsilon > 0$ such that $||P ||_{\infty, K_\epsilon} \leq 1/\sqrt{2}||P ||_{\infty, U(1)\cdot K}$ for all $P \in H^0(\VV_{D_YX}(\mcL'), \mcO_\VV) $. Then, by using the algebra structure, we get $||\cdot ||_{\infty, K_\epsilon} \leq ||\cdot ||_{\infty, U(1)\cdot K}$. And this proves that $K_\epsilon \subset \widehat{U(1)\cdot K}$.\\
Hence, this proves that $\widehat{U(1)\cdot K}$ contains a neighborhood of the zero section, that is, that $\overline{\mcL'}$ is uniformly definite.\\
\end{proof}

\section{Conservation of the arithmetic Hilbert invariants by deformation}

One the main feature of the deformation to the projective completion of the cone is the conservation of the Hilbert invariants.\\
The following proposition extracts the main hypothesis needed to prove a Hilbert invariants conservation theorem. \\
In particular, this could give a glimpse on what could be the notion of flatness in Arakelov geometry. \\

\begin{prop}\label{invarianceThm}
Let $i : Y \to \PP_\ZZ^N\times \PP_\ZZ^1$ be a closed immersion, such that $Y$ and $Y_{|1}$ the fiber over $1 \in \PP_\ZZ^1(\ZZ)$ are reduced. Let $r$ be the absolute dimension of $Y_{|1}$. Let $p:Y\to \PP_\ZZ^1$ be the projection to $\PP_\ZZ^1$.\\
Assume the following:
\begin{itemize}
    \item[i)] $\mcL = i^*(\mcO_{\PP_\ZZ^N\times \PP_\ZZ^1}(1,1))$ is endowed with a structure of semipositive seminormed line bundle $\overline{\mcL}= (\VV_Y(\mcL),T)$.
    \item[ii)] There is an action of $\GG_m$ on $Y$ such that $p$ is $\GG_m$-equivariant for the natural action of $\GG_m$ on $\PP^1_\ZZ$ and such that this action extends to an action $(G_m, U(1))$ on $\VV_Y(\overline{\mcL})$.
    \item[iii)] The following restriction map is an isomorphism : $$H^0(\VV_Y(\overline{\mcL}),\mcO_\VV)\to H^0(\VV_Y(\mcL),U(1)\cdot T_{|1} ,\mcO_\VV)$$
    \item[iv)] For $n\gg 0$, the action of $\GG_m$ induces a decomposition of $H^0(Y,\mcL^{\otimes n})$ (resp. $H^0(Y_{|\infty},\mcL^{\otimes n})$) into a sum of isotypic subspaces $\bigoplus_{l\leq n} E_{l,n}$ (resp. $\bigoplus_{0\leq l\leq n} F_{l,n}$) where $\GG_m$ acts via the character $\lambda \to \lambda^{n+l}$. And the restrictions $r_1$ to the fiber above $1$ and $r_\infty$ to the fiber above $\infty$ induces commutative diagrams:
    $$\xymatrix{
    E_{l,n} \ar[rd]^{r_\infty} \ar[r]^{r_1}  & r_1(E_{l,n}) \ar[d]^{\phi_{l,n}} \\
     & F_{l,n}
  }
$$
such that $r_1(E_{0,n}) = H^0(Y_{|1},\mcL^{\otimes n})$, $r_1(E_{l+1,n}) \subset r_1(E_{l,n}) $, where $\phi_{l,n}$ induce isomorphisms between $r_1(E_{l,n})/r_1(E_{l+1,n})  $ and $F_{l,n}$.
\end{itemize}
 
Then, $(Y_{|1}, \overline{\mcL}, \mcO_{Y_{|1}})$ and $(Y_{|\infty},\overline{\mcL},\mcO_{Y_{|\infty}})$ have the same arithmetic Hilbert invariants.  
\end{prop}

\begin{proof}
Firstly,  by the geometric Hilbert-Samuel theorem and by point iv), $Y_{|1}$ and $Y_{|\infty}$ have the same dimension $d$.\\
Now, let $r\geq d$.\\

By using ii) and iv), by the stability of the tube by the natural transverse action of  $U(1)$ and by the action of $U(1)$ on the fibers of $\VV_{Y}(\overline{\mcL})(\CC)$, 
the bornology on $H^0(\VV_{Y}(\overline{\mcL}),\mcO_{\VV}) = \bigoplus_{n,l}E_{l,n}$ can be defined by a decreasing family of Hilbertian seminorms $(||\cdot||_j)$ orthogonal for this direct sum and invariant by conjugation, by \cite{Bost-Charles1} . \\

Using the semipositivity of $\overline{\mcL}$, as $\VV_{Y}(\overline {\mcL})$ is mod-Stein, the restriction to the fiber at infinity $H^0(\VV_{Y}(\overline{\mcL}), \mcO_\VV)\to H^0(\VV_{Y_{|\infty}}(\overline{\mcL}), \mcO_\VV) $ is a strict morphism of arithmetic Hilbertian $\mcO_{\Spec \ZZ}$-modules, with finite dimensional cokernel \cite{Bost-Charles}.\\
Hence the bornology on $H^0(\VV_{Y_{|\infty}}(\overline{\mcL}), \mcO_\VV) = \bigoplus_{n,l\geq 0} F_{l,n} $ is given by the decreasing family of quotient norms $(||\cdot||_{j,q})$.\\
We have, for $n\gg 0$, $j\in \NN$,  $$ \widehat{\chi}(H^0(Y_{|\infty},\mcL^{\otimes n}),||\cdot||_{j,q} ) = \sum_{0\leq l\leq n} \widehat{\chi}(F_{l,n}, ||\cdot||_{j,q})$$

Furthermore, let's take $V_{i, \epsilon}= \{e^{-\epsilon}< |z|<e^{\epsilon}\} \cdot U_i   $ a family of relatively compact neighborhood of $U(1)\cdot T_{|1} $ in $\VV_{Y_{|1}\times \PP_\ZZ^1-\{0,\infty\}}(\mcL)(\CC)$ where $(U_i)_i$ forms a basis of relatively compact neighborhood of $T_{|1}$ in $\VV_{Y_{|1}}(\mcL)(\CC)$. Then, using point iii), the bornology on $H^0(\VV_{Y}(\overline{\mcL}), \mcO_\VV)$ is induced by the family of sup norms $||\cdot||_{V_{i,\epsilon},\infty}$, by proposition \ref{reducedcasenomrlp}.\\
Let $s\in E_{l,n}$ then, by using the action of $\GG_m$ on $s$, we get $$||s||_{V_{i,\epsilon},\infty} = ||r_1(s)||_{U_i, \infty}  e^{(n+l)\epsilon}$$
Let $(||\cdot||'_k)$ a decreasing family of Hilbertian seminorms invariant by conjugation defining the bornology on $H^0(\VV_{Y_{|1}}(\overline{\mcL}), \mcO_\VV)$. \\

Now, fix $j$. By comparing the two families of seminorms defining the bornology on $H^0(\VV_Y(\overline{\mcL}), \mcO_\VV)$, there are  $(i_0, \epsilon_0)$ and $C_0 > 0$ , such that for $i\geq i_0$ and $\epsilon \leq \epsilon_0$, we have $$C_0||\cdot||_{V_{i,\epsilon},\infty } \leq ||\cdot||_j$$on $H^0(\VV_{Y}(\mcL),\mcO_{\VV})$.\\
Furthermore, on $F_{l,n}$, $||\cdot||_{j,q}$ is the sub-quotient norm, by orthogonality of $||\cdot||_j$ on the direct sum $\bigoplus E_{l,n}$. That is $||\cdot||_{j,q}$ is given by restricting $||\cdot||_{j}$ to $E_{l,n}$ and taking the quotient on $F_{l,n}$.\\

Hence, by restriction to $E_{l,n}$ and quotient, using point iv) : we have for $i\geq i_0$, $\epsilon \leq \epsilon_0$, on $F_{l,n}$
 $$C_0    \phi_\star||\cdot||_{U_i,\infty, sq }e^{(n+l)\epsilon} \leq ||\cdot||_{j,q}$$
 where $||\cdot||_{U_i,\infty, sq }e^{(n+l)\epsilon}$ is the sub-quotient norm on $r_1(E_{l,n})/r_1(E_{l+1,n})$.\\
 
Then, by comparing the two families of seminorms defining the bornology on $H^0(\VV_{Y_{|1}}(\overline{\mcL}), \mcO_\VV)$, there are $k_0$ and $C_1>0$ such that for $k\geq k_0$ we have 
$$ C_1 ||\cdot||'_k \leq ||\cdot||_{U_{i_0},\infty}$$
Hence, for $n\gg 0$ and $0 \leq l\leq n$,  $$\widehat{\chi}(F_{l,n}, ||\cdot||_{j,q})\leq \widehat{\chi}(F_{l,n}, C_0C_1e^{(n+l)\epsilon}\phi_\star||\cdot||'_{k,sq})$$
That is, $$\widehat{\chi}(F_{l,n}, ||\cdot||_{j,q})\leq \widehat{\chi}(r_1(E_{l,n})/r_1(E_{l+1,n}), C_0C_1e^{(l+n)\epsilon}||\cdot||'_{k,sq})$$

Summing over $l$,
$$ \widehat{\chi}(H^0(Y_{|\infty},\mcL^{\otimes n}),||\cdot||_{j,q} ) \leq \sum_{0\leq l\leq n}\widehat{\chi}(r_1(E_{l,n})/r_1(E_{l+1,n}),C_0C_1e^{(n+l)\epsilon} ||\cdot||'_{k,sq})$$
That is, as $(r_1(E_{l,n}))_{0\leq l \leq n}$ forms a filtration of $H^0(Y_{|1},\mcL^{\otimes n})$, using corollary \ref{additivfiltration},
$$ \widehat{\chi}(H^0(Y_{|\infty},\mcL^{\otimes n}),||\cdot||_{j,q} ) \leq \widehat{\chi}(H^0(Y_{|1},\mcL^{\otimes n}), ||\cdot||'_{k}) -\mathrm{rk}(H^0(Y_{|1},\mcL^{\otimes n}))\log C_0C_1  -\sum_{0\leq l\leq n}\mathrm{rk}(r_1(E_{l,n})/r_1(E_{l+1,n}))(n+l)\epsilon $$
Dividing by $n^r$, taking the limits over $n$, using the fact that $\mathrm{rk}(H^0(Y_{|1},\mcL^{\otimes n}))$ is a polynomial in $n$ of degree less than $r-1$, and that the last summand is negative, we get 
$$\liminf_n \frac{r!}{n^r}\widehat{\chi}(H^0(Y_{|\infty},\mcL^{\otimes n}),||\cdot||_{j,q} ) \leq \liminf_n \frac{r!}{n^r} \widehat{\chi}(H^0(Y_{|1},\mcL^{\otimes n}), ||\cdot||'_{k})  $$
Then using the monotony of $\widehat{\chi}$ to take $k$ to infinity and finally $j$ to infinity we get  
$$\underline{c}_r(Y_{|\infty}, \overline{\mcL}, \mcO_{Y_{|\infty}}) \leq 
\underline{c}_r(Y_{|1}, \overline{\mcL}, \mcO_{Y_{|1}}) $$

Similarly, fix $k$ and $\epsilon>0 $. \\
Then, by comparing the two families of seminorms defining the bornology on $H^0(\VV_{Y_{|1}}(\overline{\mcL}), \mcO_\VV)$, there are $i_0$ and $C'_1>0$ such that for $i\geq i_0$ we have 
$$ C'_1 ||\cdot||_{U_i,\infty} \leq ||\cdot||'_k$$
By comparing the two families of seminorms defining the bornology on $H^0(\VV_Y(\overline{\mcL}), \mcO_\VV)$, there are $j_0$ and $C_0'>0$ such that for $j\geq j_0$ we have $$C'_0||\cdot||_j \leq ||\cdot||_{V_{i_0,\epsilon},\infty }$$
Using that $\sum_{0\leq l\leq n}\mathrm{rk}(r_1(E_{l,n})/r_1(E_{l+1,n})(n+l)\leq 2n \mathrm{rk}(H^0(Y_{|1},\mcL^{\otimes n}))$, we get that
$$\liminf_n \frac{r!}{n^r}\widehat{\chi}(H^0(Y_{|\infty},\mcL^{\otimes n}),||\cdot||_{j,q} ) \geq \liminf_n \frac{r!}{n^r}\left( \widehat{\chi}(H^0(Y_{|1},\mcL^{\otimes n}), ||\cdot||'_{k})  -2n \mathrm{rk}(H^0(Y_{|1},\mcL^{\otimes n})\right)\epsilon $$

Using the monotony of $\widehat{\chi}$ to take $j$ to infinity and then $\epsilon$ to $0$ and $k$ to infinity, we get  
$$\underline{c}_r(Y_{|\infty}, \overline{\mcL}, \mcO_{Y_{|\infty}}) \geq 
\underline{c}_r(Y_{|1}, \overline{\mcL}, \mcO_{Y_{|1}}) $$
Hence, 
$$\underline{c}_r(Y_{|\infty}, \overline{\mcL}, \mcO_{Y_{|\infty}}) = 
\underline{c}_r(Y_{|1}, \overline{\mcL}, \mcO_{Y_{|1}}) $$
By replacing $\liminf$ by $\limsup$ in the previous inequalities, we also get that $$\overline{c}_r(Y_{|\infty}, \overline{\mcL}, \mcO_{Y_{|\infty}}) = 
\overline{c}_r(Y_{|1}, \overline{\mcL}, \mcO_{Y_{|1}})$$
This concludes.
\end{proof}

Let $X$ be a projective scheme over $\ZZ$. Let $\overline{\mcL}$ be a semipositive seminormed line bundle over $X$, very ample over $\ZZ$. Let $i:Y\to X$ be a hyperplane section with respect to $\mcL$.\\
Let $(D_YX, \overline{\mcL'})$ be the deformation to the projective completion of the cone associated to $(X,Y,\overline{\mcL})$.\\

$(D_YX, \overline{\mcL'})$ gives a deformation from $(X, \overline{\mcL})$ to $(D_YX_{|\infty}, \overline{\mcL'})$ and 
we prove that  $(X, \overline{\mcL}, \mcO_X)$ and $(D_YX_{|\infty}, \overline{\mcL'},\mcO_{D_YX_{|\infty}} )$ share the same arithmetic Hilbert invariants.\\
For the purpose of verifying point iii) of proposition \ref{invarianceThm}, we need the following lemma : \\

\begin{lemma}\label{inducedBornology} Let $X$ be a reduced projective scheme over $\ZZ$. Let $\overline{\mcL}=(\VV_X(\mcL),T)$ be a semipositive seminormed line bundle over $X$, very ample over $\ZZ$. Assume furthermore that the metric is uniformly definite on $\overline{\mcL}$. Let $i:Y\to X$ be a hyperplane section with respect to $\mcL$.\\
Let $(D_YX, \overline{\mcL'})$ be the deformation to the projective completion of the cone associated to $(X,Y,\overline{\mcL})$.\\
Then the bornology on $H^0(\VV_{D_YX}(\overline{\mcL'}), \mcO_\VV)$ is given by the restriction map to the orbit of $T$ under $U(1)$, i.e. the following map is an isomorphism: $$H^0(\VV_{D_YX}(\overline{\mcL'}), \mcO_\VV)\to H^0(\VV_{D_YX}(\mcL'),U(1)\cdot T, \mcO_\VV) $$
\end{lemma}

\begin{proof} Let $0_{D_YX}$ denote the zero section in $\VV_{D_YX}(\mcL')$.
As $U(1)\cdot T \cap 0_{D_YX} \neq \emptyset$ and as $0_{D_YX}$ is a compact subspace of $\VV_{D_YX}(\mcL')(\CC)$, we have $\widehat{U(1)\cdot T \cup 0_{D_YX}} = \widehat{U(1)\cdot T}$.\\
Then, as $U(1)\cdot T \cup 0_{D_YX} $ is a seminorm on the line bundle $\mcL'$ over $D_YX$, by proposition \ref{passageAlenveloppeholo}, the  restriction map is an isomorphism $H^0(\VV_{D_YX}(\overline{\mcL'}), \mcO_\VV)\to H^0(\VV_{D_YX}(\mcL'),U(1)\cdot T \cup 0_{D_YX}, \mcO_\VV)$.\\
As $X$ is reduced, $D_YX$ is reduced. By \cite{Bost-Charles}, this implies that the bornology can be given by a family of sup norms on a basis of relatively compact neighborhoods of $U(1)\cdot T \cup 0_{D_YX}$.\\
Now, we can construct such a basis of neighborhood the following way.\\  
Let $U$ be a relatively compact neighborhood of $U(1)\cdot T$. By proposition \ref{stongdeftostrongdef}, $\widehat{U(1)\cdot T} = \widehat{U(1)\cdot T \cup 0_{D_YX}}$ is a uniformly definite metric on $\mcL'$. Hence, it is a neighborhood of $0_{D_YX}$. Let $V$ a relatively compact neighborhood of $0_{D_YX}$ such that $\overline{V}$ is in the interior of $\widehat{U(1)\cdot T \cup 0_{D_YX}}$. Then, $U\cup V$ is a relatively compact neighborhood of $U(1)\cdot T \cup 0_{D_YX}$.\\
In this case, we have $||\cdot||_{\infty, U\cup V} = ||\cdot||_{\infty, \widehat{U\cup V}} \geq ||\cdot||_{\infty, \widehat{U(1)\cdot T \cup 0_{D_YX}}} > ||\cdot||_{\infty, V}$, by the maximum principle.\\
Hence, $||\cdot||_{\infty, U \cup V} = ||\cdot||_{\infty, U }$. And the bornology is given by a family of sup norms on a basis of relatively compact neighborhoods of $U(1)\cdot T$.

\end{proof}

\begin{thm}\label{invarianceArithmeticHilbert}Let $X$ be a reduced projective scheme over $\ZZ$. Let $\overline{\mcL}=(\VV_X(\mcL),T)$ be a semipositive seminormed line bundle over $X$, very ample over $\ZZ$. Assume furthermore that the metric is uniformly definite on $\overline{\mcL}$. Let $i:Y\to X$ be a hyperplane section with respect to $\mcL$.\\
Let $(D_YX, \overline{\mcL'})$ be the deformation to the projective completion of the cone associated to $(X,Y,\overline{\mcL})$.\\
Then $X$ and $D_YX_{|\infty}$ have the same dimension $d$.\\
 $(X, \overline{\mcL}, \mcO_X)$ and $(D_YX_{|\infty}, \overline{\mcL'},\mcO_{D_YX_{|\infty}} )$ have the same arithmetic Hilbert invariants.\\
\end{thm}

\begin{proof} 
The result is a direct application of proposition \ref{invarianceThm} applied to $D_YX$. If $X$ is reduced, then it is clear from the definition that $D_YX$ is reduced. Point i) is verified by definition \ref{tubeDefinition} of the analytic tube in $\VV_{D_YX}(\mcL')(\CC)$. Point ii) is verified by the definition \ref{definitionActionGm} of the action of $\GG_m$ and proposition \ref{stabilityTube}. Point iii) is lemma \ref{inducedBornology}. Point iv) is proposition \ref{decompSelonActionGm}.\end{proof}

\section{\textsc{Appendix}}

\subsection*{Relationship between the deformation to the projective completion of the cone and the deformation to the normal cone.}
The relationship with the deformation to the normal cone can be spotted in the exposition in \cite{fulton} of the deformation to the normal cone.\\
Indeed, it describes the deformation to the normal cone as a local construction which takes a closed immersion of affine schemes $\Spec B/J \to \Spec B$ to the associated deformation $\Spec \bigoplus_{l\in \ZZ} J^l(t/u)^{-l}$.\\
If we consider the associated deformation $M^\circ_YX$ over $\PP^1_A$, which is the previous definition extended over $\{u \neq \infty\} = \Spec A[u/t]$ by $X\times_A \AA^1_A$.
This definition immediately yields the following proposition:\\

\begin{prop}\label{openNormalcone} Let $D_YX$ be the deformation to the projective completion of the cone associated to $(X,Y,\mcL)$, where $Y$ is a hyperplane section with respect to the line bundle $\mcL$ very ample over $A$. Let $M^\circ_YX$ be the deformation to the normal cone over $\PP^1_A$.\\
Then, there is an open immersion $M^\circ_YX \to D_YX$ over $\PP^1_A$ which is an isomorphism over $\{u\neq \infty\}$.
\end{prop}

\begin{proof}
Let $S_\bullet = \bigoplus_{n\in \NN} H^0(X,\mcL^{\otimes})$, $I_\bullet \subset S_\bullet$ be the graded ideal associated to the closed $i:Y\to X$. \\
Over $\{u \neq 0\}$, $D_YX$ is defined by $D_YX_{u\neq 0} = \Proj S'_\bullet$ where $S'_{\bullet }$ is the graded $A[t/u]$-algebra   $ \bigoplus\limits_{l\in \ZZ} I_{\bullet}^l(t/u)^{-l}$. \\
Now take a non-zero section $s$ in $H^0(X,\mcL)$, then $s\in S'_{1}$, and the associated open subscheme is $\Spec \bigoplus\limits_{l\in \ZZ} J^l(t/u)^{-l} $ where $J$ is the localization of $I_{\bullet}$ at $s$ taken in degree $0$. That is, $J$ is the ideal associated to the pullback of the closed immersion $i:Y\to X$ above the open affine subscheme of $X$ defined by $\{s\neq 0\}$.\\
Hence, the open subscheme associated to $s\in S'_{1}$ is the open subscheme of $M_Y^\circ X$ above the open affine subscheme of $X$ defined by $\{s\neq 0\}$.\\
The claim is then straightforward.
\end{proof}

\begin{rem}
As a corollary, there is a locally closed immersion above $\PP^1_A$ from $Y\times_A \PP^1_A$ to $D_YX$.\\
\end{rem}

The following proposition gives an embedding to $\PP^{N+M}_A \times_A \PP^1_A$ and another relation with the deformation to the normal cone.  \\

\begin{prop}\label{immersionPnn} Let $X$ be a projective scheme over $A$. Let $\mcL$ be a line bundle over $X$ very ample over $A$. Let $i:Y\to X$ be a closed immersion where the associated graded ideal $I_{\bullet}\subset \bigoplus H^0(X,\mcL^{\otimes n})$ is generated in degree $1$ by global sections $t_1,\dots,t_M$ of $\mcL$.\\
Assume that $s_0, \dots, s_N \in H^0(X,\mcL)$ induce a closed embedding into $\PP^N_A$.\\
Then, $D_YX$ is the schematic closure of the image of the map:$$\begin{array}{ccccc}
\varphi & :   X \times_A \AA^1_A   & \to & \PP^{N+M}_A \times_A \PP^1_A \\
  & (x , u) & \mapsto & ([s_0(x) : \dots : s_N(x):u t_1(x) : \dots ut_M(x)], [u:1]) 
  \end{array}$$
\end{prop}

\begin{proof}
This is proposition \ref{immersionPn}.
\end{proof}

In the case of a closed immersion $i:Y \to X$, where the associated graded ideal $I_{\bullet}\subset \bigoplus H^0(X,\mcL^{\otimes n})$ is generated in degree $1$ by global sections $t_1,\dots,t_M$ of $\mcL$, the construction in \cite{fulton}[Chapter 5] or in \cite{Baum1975RiemannrochFS} describe the deformation to the normal cone by considering the schematic closure $M_YX$ of the map 
$$\begin{array}{ccccc}
\psi & :   X \times_A \AA^1_A   & \to & \PP_X(\mcO_X \oplus \mcL \oplus \dots \oplus \mcL)\times_A \PP^1_A \\
  & (x , u) & \mapsto & (x,[1 : u t_1(x) : \dots : ut_M(x)], [u:1]) \\
\end{array}$$ which has two components over the fiber above infinity. One of them if the blow up $\Tilde{X}_Y$ of $X$ along $Y$, and the deformation to the normal cone $M^\circ_YX$ is the complement to $\Tilde{X}_Y$ in $M_YX$. 

As the proposition above suggests, this embeddings give another relation with the deformation to the normal cone:\\
\begin{prop} Let $X$ be a projective scheme over $A$. Let $\mcL$ be a line bundle over $X$ very ample over $A$. Let $i:Y\to X$ be a closed immersion where the associated graded ideal $I_{\bullet}\subset \bigoplus H^0(X,\mcL^{\otimes n})$ is generated in degree $1$ by global sections $t_1,\dots,t_M$ of $\mcL$.\\
Assume that $s_0, \dots, s_N \in H^0(X,\mcL)$ induce a closed immersion into $\PP^N_A$.\\
Then, there is a map from $\PP_X(\mcO_X \oplus \mcL \oplus \dots \oplus \mcL) $ to $\PP^{N+M}_A$ which induces a map $\nu $ from $M_YX$ to $D_YX$.\\
Furthermore, $\nu$ restricts to the open immersion $M^\circ_YX \to D_YX$ described in proposition \ref{openNormalcone}.
\end{prop}

\begin{proof}
Firstly, we have an isomorphism between $\PP_X(\mcO_X \oplus \mcL \oplus \dots \oplus \mcL) $ and $\PP_X(\mcL^{\vee}\oplus \mcO_X \oplus \dots \oplus \mcO_X)$. Then the closed immersion from $X$ to $\PP^N_A$ induces a morphism from  $\PP_X(\mcL^{\vee}\oplus  \mcO_X \oplus \dots \oplus \mcO_X)$ to $\PP_{\PP^N_A}(\mcO_\PP(-1) \oplus \mcO_{\PP} \oplus \dots \oplus \mcO_{\PP})$, as $\mcL^{\vee}\oplus  \mcO_X \oplus \dots \oplus \mcO_X$ is the pullback of $\mcO_\PP(-1) \oplus \mcO_{\PP} \oplus \dots \oplus \mcO_{\PP}$. Then, there is a map $\PP_{\PP^N_A}(\mcO_\PP(-1) \oplus \mcO_{\PP} \oplus \dots \oplus \mcO_{\PP}) \to \PP_{A}^{N+M}$ by choosing coordinates for $\mcO_\PP(-1)$. \\
This map induces a commutative diagram
$$\xymatrix{
    X\times_A \AA^1_A \ar[rd] \ar[r]  & \PP_X(\mcO_X \oplus \mcL \oplus \dots \oplus \mcL)\times_A \PP^1_A \ar[d] \\
     & \PP^{N+M}_A \times_A \PP^1_A
  }
$$
More concretely, $(x,u ) \in X \times_A \AA^1_A$ is sent to $(x,[1 : u t_1(x) : \dots : ut_M(x)], [u:1]) \in \PP_X(\mcO_X \oplus \mcL \oplus \dots \oplus \mcL)\times_A \PP^1_A$ which is sent to $(x,[\varphi_x : u \varphi_x(t_1(x)) : \dots : u \varphi_x(t_M(x))], [u:1]) \in \PP_X(\mcL^{\vee}\oplus \mcO_X\oplus \dots \oplus \mcO_X) $ where $\varphi_x \in \mcL^{\vee}_x-\{0\}$, which, via the identification of $\varphi_x$ to $\lambda(s_0(x),\dots, s_n(x))$, is sent to $([s_0(x): \dots : s_n(x) : u t_1(x) : \dots : ut_M(x)], [u:1]) $.  \\
Taking the schematic closure for $\phi$ and $\varphi$ gives the map from $M_YX$ to $D_YX$.\\
To see the last point, it suffices to choose $s \in H^0(X,\mcL)$ and to localize the diagram to this $\{s\neq 0\}$.\\
For simplicity, up to a change of variable, we can choose $s=s_0$.
Then, it suffices to look at the restricted diagram:
$$\xymatrix{
    X_{s_0\neq 0}\times_A \AA^1_A \ar[rd] \ar[r]  & \VV_{X_{s_0\neq 0 }}(\mcL^\vee \oplus \dots \oplus \mcL^\vee)\times_A \PP^1_A \ar[d] \\
     & {\PP^{N+M}_A}_{X_0\neq 0} \times_A \PP^1_A
  }
$$
$X_{s_0\neq 0}$ is affine, i.e. $ X_{s_0\neq 0} =\Spec B$ and $\mcL$ is trivial above $X_{s_0\neq 0}$.\\
Above ${u\neq 0}$, the previous diagram corresponds to a diagram of $\ZZ[t/u]$-algebra 
$$\xymatrix{
    B[t/u]   & B[y'_1,\dots , y'_M][t/u] \ar[l] \\
     & \ZZ[x_1/x_0, \dots, x_N/x_0, y_1/x_0,\dots , y_M/x_0][t/u] \ar[ul] \ar[u] 
  }
$$
where the vertical map sends $x_i/x_0$ to $s_i/s_0$ and $y_j/x_0$ to $y'_j/s_0$ and the horizontal map sends $y'_j$ to $t_j(t/u)^{-1}$.\\
The image of both maps into $B[t/u]$ are the same : the algebra defining the deformation to the normal cone above $\{s_0 \neq 0\}$.\\
This concludes.
\end{proof}

\begin{ex}
In the case of $Y = div(s)$ where $s$ is a regular global section of $\mcL$, the blow-up of $X$ along $Y$ is $X$ itself. And the map described above from $M_YX$ to $D_YX$ sends the copy of $X$ above infinity (corresponding to $(x,[0:1],[0:1]), x\in X$) to the point $([0:\dots : 0:1],[0:1])$ in $\PP^{N+1}_A\times_A \PP^1_A$. 

\end{ex} 


\nocite{*}
\printbibliography

\end{document}